\documentclass[12pt]{article}

\usepackage{amsmath,amssymb,amsthm,xcolor,mathrsfs,enumitem,mathtools,centernot}

\usepackage[margin=0.9in]{geometry} 
\usepackage{authblk}

\usepackage[hidelinks, backref=page]{hyperref}
\hypersetup{
    colorlinks,
    citecolor=blue,
    filecolor=blue,
    linkcolor=blue,
    urlcolor=blue
    }
\usepackage[ocgcolorlinks]{ocgx2}
\usepackage[T1]{fontenc}
\theoremstyle{definition}
\newtheorem{defn}{Definition}[section]

\newtheorem{thm}[defn]{Theorem}
\newtheorem{cor}[defn]{Corollary}

\newtheorem{rem}[defn]{Remark}
\newtheorem{ex}[defn]{Example}

\numberwithin{equation}{section}

\newcommand{\amst}{\mathsf{amst}}
\newcommand{\Th}{\mathsf{Th}}
\newcommand{\Mod}{\mathsf{Mod}}
\newcommand{\pow}{\mathcal{P}}
\newcommand{\lang}{\mathscr{L}}
\newcommand{\mstr}{\mathfrak{M}}
\newcommand{\Finset}{\mathsf{FinSet}}
\newcommand{\modfin}{\mathsf{ModFin}}
\newcommand{\thmref}[1]{Theorem~\ref{#1}}
\newcommand{\secref}[1]{\S\ref{#1}}

\newcommand{\subsecref}[1]{\S\S\ref{#1}}
\newcommand{\defref}[1]{Definition~\ref{#1}}
\newcommand{\remref}[1]{Remark~\ref{#1}}
\newcommand{\corref}[1]{Corollary~\ref{#1}}
\newcommand{\exref}[1]{Example~\ref{#1}}

\DeclareRobustCommand\deduc{\mathrel{|}\joinrel\mkern-.5mu\mathrel{-}}
\title{Abstract Model Structures and Compactness Theorems}
\author[1]{Sayantan Roy}
  \author[2]{Sankha S. Basu }
  \author[3]{Mihir K. Chakraborty}
  \affil[1]{National Research University Higher School of Economics, International Laboratory for Logic,
Linguistics and Formal Philosophy, Moscow, Russia}
  \affil[2]{Department of Mathematics, Indraprastha Institute of Information Technology-Delhi, New Delhi, India}
  \affil[3]{School of Cognitive Science, Jadavpur University, Kolkata, India}
\date{August 09, 2026}

\begin{document}

\maketitle
\begin{abstract}
    The compactness theorem for a logic states, roughly, that the satisfiability of a set of well-formed formulas can be determined from the satisfiability of its finite subsets, and vice versa. Usually, proofs of this theorem depend on the syntactic/semantic particularities of the corresponding logic. In this paper, using the notion of \emph{abstract model structures}, we show that one can develop a generalized notion of compactness that is independent of these. Several characterization theorems for a particular class of compact abstract model structures are also proved. 
\end{abstract}

\textbf{Keywords:}
Universal Logic, Universal Model Theory, Abstract Model Structures, Compactness
\tableofcontents
\section{Introduction}
Model theory studies the relationships between the syntax and the semantics of (formal) languages via mathematical structures. This amounts to the study of structures of the form $\mstr=(\mathbf{M},\models,\pow(\lang))$, where $\mathbf{M} (\ne\emptyset)$ and $\lang$ are sets, $\pow(\lang)$ is the power set of $\lang$, and $\models\,\subseteq\,\mathbf{M}\times \pow(\lang)$. One can think of $\lang$ as the set of \emph{sentences} or \emph{well-formed formulas (wffs)} of some language, $\mathbf{M}$ as a set of \emph{structures} (or \emph{models}) for $\lang$, and $\models$ as the \emph{satisfaction relation} between sets of sentences and models. In this paper, a structure of the above form --- i.e. a structure of the form $(\mathbf{M},\models,\pow(\lang))$ --- will be referred to as an \emph{abstract model structure} ($\mathsf{amst}$).

The compactness theorem for first-order logic is one of the central theorems of model theory. Together with the downward Löwenheim-Skolem theorem, it characterizes the expressive power of standard first-order logic among all logics which obey certain regularity conditions, the formal statement of which is known as \emph{Lindstr{\"o}m's theorem} (see \cite{ChangeKeisler1990, Lindstrom1969} for details). A logic is said to be \emph{compact} if, for every set $\Gamma$ of wffs, it is satisfiable iff every finite subset of $\Gamma$ is so. Therefore, in order to understand the compactness theorem for any logic, it is crucial to understand what `satisfiability' (of an arbitrary set of wffs) means in the context of \emph{that} logic.  However, the notion of satisfiability can also be defined independent of any underlying logic, as follows.

\hypertarget{def:comp}{\begin{defn}[\textsc{Satisfiability, Finite Satisfiability and Compactness}]{\label{def:comp}}
Given an $\mathsf{amst}$ $\mstr=(\mathbf{M},\models,\pow(\lang))$, $m\in \mathbf{M}$, and a set $\Gamma\subseteq \lang$, we say that $m$ \emph{satisfies} $\Gamma$ \emph{in $\mstr$} if $m\models \Gamma$. A set $\Gamma\subseteq \lang$ is said to be \emph{satisfiable in $\mstr$} if there exists $m\in \mathbf{M}$ such that $m$ satisfies $\Gamma$, and \emph{finitely satisfiable in $\mstr$} if every finite subset of $\Gamma$ is satisfiable. An $\mathsf{amst}$ $\mstr=(\mathbf{M},\models,\pow(\lang))$ is said to be \emph{compact} if for all $\Gamma\subseteq \lang$, $\Gamma$ is satisfiable in $\mstr$ iff it is finitely satisfiable in $\mstr$.
\end{defn}}
We will often drop an explicit mention of $\mstr$, for discussions involving a single $\amst$. 

Now, we ask the following question. Is it possible to characterize compact $\amst$s?
This article is dedicated to answering this question for a special class of $\amst$s that we call \emph{normal $\mathsf{amst}$}s (see \hyperlink{def:normamst}{\defref{def:normamst}}). Before going into the technical details, we point out below some characterizing features of the spirit of this article by comparing and contrasting it with some earlier works of similar flavor. 

\begin{enumerate}[label=(\arabic*)]
    \item Although we are trying to characterize compact normal $\mathsf{amst}$s, we do so without an underlying logical system. Thus, even though in essence, our results may be said to belong to the subject known as \emph{abstract model theory} (in the sense of \cite{Barwise1974, Barwise1985}), they lack the strong commitment to the conventional concrete systems of logic that are characteristic features of works in this area. Instead, the results of this paper may be seen as an extension and/or continuation of the framework of abstract model theory proposed in \cite{Garcia-Matos_Vaananen2005}, since the notion of an `abstract logic' defined there corresponds to what we call a \emph{normal} $\amst$. Moreover, our results are also independent of precise notions of concrete models, sentences, satisfaction and so on. Thus, the spirit of this paper is similar to that of `universal model theory' in the following sense (see \cite{Diaconescu2007}). 
    \begin{quote}
        \begin{small}
        [U]niversal model theory [is] part of the grand project of universal logic promoted by B{\'e}ziau and others \cite{Beziau2006}. Like universal logic in general, universal model theory in particular is not seeking for one single model theory in which all other model theories can be expressed. This would be an unrealistic approach based upon an essentialist view on logic and model theory. The universal model theory ideal rather means a mathematical framework for developing model-theoretic concepts and results in a non-essentialist and groundless manner, and which may be reflected at the level of actual logics in the form of concrete model theories. Here by ‘non-essentialist’ we mean the absence of a lasting, individual essence of anything in general, and of logics and logical theories in particular, and by ‘groundless’ we mean a true absence of commitment to actual logical systems.
        \end{small}
    \end{quote}
    \item Our mathematical approach here is set-theoretic and not --- as, e.g., is the case in \cite{AndrékaNémeti1983, Diaconescu2008, Ehresmann1968, GuitartLair1980, LambekScott1986, MakkaiPare1989, MakkaiReyes1977} --- category-theoretic. This choice is motivated by the following considerations. First, from a purely pedagogical perspective, the set-theoretic approach is arguably simpler, and so has greater potential to reach a larger audience. Second, since our results are also language-independent, we are not compromising with the spirit of `universal model theory' in terms of abstraction, and hence in terms of applicability. Third, in all the references above, except for \cite{Diaconescu2008} which adopts the framework of \emph{institutions}, the satisfaction relation is \emph{defined}, and not \emph{axiomatized}. However, we also do not adopt this institution-theoretic approach mainly because of the following reasons. 
\begin{enumerate}[label=$\bullet$]
	\item The semantic consequence relation in any institution can be proven to satisfy reflexivity, monotonicity and transitivity (\cite[Proposition 3.7]{Diaconescu2008}). This is also the case with our treatment of normal $\amst$s. Thus, we do not lose any of the Tarskian conditions by taking this route.
	\item  A comparison with \cite[Section 6.4]{Diaconescu2008}, where \textit{model compactness} of institutions is discussed, convinces us that more general results can be phrased without the jargon of institution theory, and in particular, of category theory. Whether this is true for the rest of the results of institution-independent model theory is another matter entirely and will not be pursued here.
	\item Institutions are only language-independent in the sense that they do not consider any \textit{particular} signature of any logic; however, they \textit{do} consider signatures in an abstract way. In contrast, we do not consider signatures at all.
	\item While the translational invariance of truth, encoded in the satisfaction condition of the institutions, is important in more concrete settings, we aimed to develop language-independent model theory without this requirement. The spirit of our paper is thus to develop the foundations of universal model theory with as few assumptions as possible so as to maximize applicability.
\end{enumerate} 
\end{enumerate}

The remainder of this article is organized as follows. In the next section, we give several examples of $\amst$s. Normal $\mathsf{amst}$ is defined in \hyperlink{sec:normamst}{\secref{sec:normamst}}. The main aim of this section is to prove several characterizations of compact normal $\mathsf{amst}$s. In \hyperlink{sec:com-compl}{\subsecref{subsec:com-compl}}, we show that every normal $\mathsf{amst}$ induces a Tarski-type logical structure and vice versa. We then prove our first characterization theorem for normal $\amst$s (\hyperlink{thm:compact_normal_ams(I)}{\thmref{thm:compact_normal_ams(I)}}). Next, in \hyperlink{sec:com-top}{\subsecref{subsec:com-top}}, generalizing the topological proof of the compactness theorem for propositional calculus (see, e.g., \cite{Paseau2010}) we provide a topological characterization of compact normal $\amst$s (\hyperlink{thm:compact_normal_ams(II)}{\thmref{thm:compact_normal_ams(II)}}). In \hyperlink{subsec:comp-ultra}{\subsecref{subsec:comp-ultra}}, the ultraproduct proof of compactness is generalized and the notion of \emph{ultramodels}, which can be thought of as a generalization of ultraproducts, is introduced. These help us prove two more characterizations of compact normal $\amst$s (\hyperlink{thm:finsat_ultramodel=>sat}{\thmref{thm:finsat_ultramodel=>sat}}, \hyperlink{thm:compact_normal_ams(IV)}{\thmref{thm:compact_normal_ams(IV)}}), a sufficient condition for the compactness of normal $\amst$s (\hyperlink{thm:sufcomp}{\thmref{thm:sufcomp}}), and a generalized version of \L o\'{s}’s theorem (\hyperlink{thm:genLoz}{\thmref{thm:genLoz}}). Finally, in the concluding section, we point out some directions for future research.   

\section{Examples}
We have already mentioned that $\amst$s can be seen as a generalized expression of the standard model-theoretic enterprise. This section brings together a number of examples that emphasize further the versatility of these.

\hypertarget{ex:logic_str}{\begin{ex}[\textsc{Logical Structures}]{\label{ex:logic_str}} A \emph{logical structure} (see \cite{Beziau1994}) is a pair $(\lang, \deduc)$, where $\lang$ is a set and $\deduc\,\subseteq\, \mathcal{P}(\lang)\times \lang$. This can be represented as the $\mathsf{amst}$ $\mstr_{\mathfrak{L}}=(\lang,\models,\pow(\lang))$, where for all $\Gamma\cup\{\alpha\}\subseteq \lang$, $\alpha\models \Gamma$ iff $\Gamma\not\deduc\alpha$. Conversely, given an $\mathsf{amst}$ of the form $\mstr=(\lang,\models,\pow(\lang))$, one can define a logical structure $(\lang,\deduc)$, where for all $\Gamma\cup\{\alpha\}\subseteq \lang$, $\Gamma\not\deduc\alpha$ iff $\alpha\models \Gamma$. 
\end{ex}}

The next example shows that an \emph{information system} can be seen as an $\amst$. The notion was introduced by Scott in \cite{Scott1982}. We, however, follow the presentation in \cite[Chapter 12]{Winskel1993}.

\begin{ex}[\textsc{Information Systems}]
    An \emph{information system} is defined to be a structure $\mathfrak{A}=(A,\text{Con},\deduc)$, where $A$ is a countable set of \emph{tokens}, $\text{Con}$ is a non-empty subset of the set of all finite subsets of $A$, called the \emph{consistent sets} and $\deduc\,\subseteq(\text{Con}\setminus\{\emptyset\})\times A$ is the \emph{entailment relation} satisfying the following properties. 
    \begin{enumerate}[label=(\alph*)]
        \item If $Y\in \text{Con}$ and $X\subseteq Y$ then $X\in \text{Con}$.
        \item If $a\in A$ then $\{a\}\in \text{Con}$.
        \item $X\deduc a$ implies that $X\cup\{a\}\in \text{Con}$.
        \item $X\in \text{Con}$ and $a\in X$ implies that $X\deduc a$.
        \item If $X,Y\in \text{Con}$, and $X\deduc b$ for all $b\in Y$, then $Y\deduc c$ implies that $X\deduc c$.
    \end{enumerate}
    Then, $\mstr_{\mathfrak{A}}=(A,\models,\pow(A))$ is an $\amst$, where $\models$ is defined as follows. For all $\Gamma\cup\{a\}\subseteq A$, $a\models \Gamma$ iff $\Gamma\deduc a$ (cf. \hyperlink{ex:logic_str}{\exref{ex:logic_str}}).
\end{ex} 

Our next example considers \emph{Chu spaces}. The generality of Chu spaces allows for a uniform handling of various mathematical structures, e.g., sets, posets, semi-lattices,  preordered sets, Stone spaces, ordered Stone spaces, topological spaces, locales, complete semilattices, distributive lattices (but not general lattices), algebraic lattices, frames, profinite (Stone) distributive lattices, Boolean algebras, and complete atomic Boolean algebras, to name a few (for details see \cite{Pratt1995a, Pratt1999}). One can also realize the notion of \emph{contexts} studied in Formal Concept Analysis (see \cite{GanterWille1999} for details), that of \emph{classifications} discussed in \cite{BarwiseSeligman1997} and that of \emph{topological systems} investigated in \cite{Vickers1989} as certain kinds of Chu spaces. 

\begin{ex}[\textsc{Chu Spaces}]
    A \emph{Chu space over a set $K$} is a triple $\mathfrak{C}=(X,r,A)$, where $X$ and $A$ are sets and $r: X\times A\to K$.
    This Chu space gives rise to the $\amst$ $\mstr_{\mathfrak{C}}=(X\times A,\models,\pow(K))$, where for all $(x,a)\in X\times A$ and $\Gamma\subseteq K$, $(x,a)\models \Gamma$ iff $\Gamma=\{k\}$ for some $k\in K$, and $r(x,a)=k$.
\end{ex}

We now show that directed multigraphs (see e.g., \cite{Bang-JensenGutin2009}), and even categories, can be thought of as certain kinds of $\amst$s.

\begin{ex}[\textsc{Directed Multigraphs}]
    A \emph{directed multigraph/multidigraph} (also called \emph{quiver}) is a quadruple $\mathfrak{Q}=(V,E,s,t)$, where 
    \begin{enumerate}[label=(\alph*)]
        \item $V$ is a set, the elements of which are called \emph{vertices} of $\mathfrak{Q}$,
        \item $E$ is a set, the elements of which are called \emph{edges} of $\mathfrak{Q}$,
        \item $s:E\to V$ which assigns to each edge $e$, the \emph{source} of $e$, viz., $s(e)$, and 
        \item $t:E\to V$ which assigns to each edge $e$, the \emph{target} of $e$, viz., $t(e)$.
    \end{enumerate}Each such directed multigraph induces an $\amst$ $\mstr_{\mathfrak{Q}}=(E,\models,\pow(V\times V))$, where for all $e\in E$ and $\Gamma\subseteq V\times V$,
    \[
    e\models \Gamma\hbox{ iff }\Gamma=\{(s(e),t(e))\}
    \]
    Conversely, suppose $\mstr$ is an $\amst$ of the form $(E,\models,\pow(V\times V))$, where $\models\,\subseteq E\times\pow(V\times V)$ is such that for each $e\in E$, there exists a unique $(a,b)\in V\times V$ such that $e\models\{(a,b)\}$. Then, $\mstr$ induces a directed multigraph $(V,E,s_{\mstr},t_{\mstr})$, where $s_{\mstr}:E\to V$ and $t_{\mstr}:E\to V$ are defined as follows. For all $e\in E$, $(s_\mstr(e),t_\mstr(e))$ is the unique element of $V\times V$ such that $e\models\{(s_\mstr(e),t_\mstr(e))\}$. 
\end{ex}

Since categories can be thought of as certain kinds of multidigraphs, it is hardly surprising that categories too can be represented as $\amst$s and vice versa. This is illustrated in the next example. The definition of a category that we employ here is modelled after the so-called ‘object-free' one as  in \cite[Definition 3.53]{AdamakHerrlichStrecker?}. The details are as follows.

\begin{ex}[\textsc{Categories}]
    The object-free definition of categories relies on the definition of \emph{partial binary algebras}. These are pairs of the form $(X,\ast)$ where $X$ is a set%
    \footnote{In the original definition \cite[Definition 3.53]{AdamakHerrlichStrecker?}, $X$ is taken to be a \emph{class}. Thus, the definition of object-free categories we use below is the same as that of \emph{small} object-free categories. While for category theory, these size issues become important, it is not relevant for the current example. One can find more details in \cite{Borceux1994, Maclane1998}.} and $\ast$ is an operation defined on a subset of $X\times X$. The value of $\ast(x,y)$ is denoted by $x\ast y$. An element $u\in X$ is said to be a \emph{unit of $(X,\ast)$} if for all $x,y\in X$, $x\ast u=x$, whenever $x\ast u$ is defined, and $u \ast y = y$, whenever $u \ast y$ is defined.

    An \emph{object-free category} is a partial binary algebra $\mathfrak{P}=(M,\circ)$, where the members
    of $M$ are called \emph{morphisms}, that satisfy the following conditions:
    \begin{enumerate}[label=($\text{\alph*}_{\mathfrak{P}}$)]
        \item \underline{Matching Condition}: For morphisms $f, g$, and $h$, the following conditions are equivalent:
        \begin{enumerate}[label=$\bullet$]
            \item $g \circ f$ and $h \circ g$ are defined,
            \item $h \circ (g \circ f)$ is defined, and
            \item $(h\circ g)\circ f$ is defined.
        \end{enumerate}
        \item \underline{Associativity Condition}: If morphisms $f, g$, and $h$ satisfy the matching conditions, then $h \circ (g \circ f) = (h \circ g) \circ f$.
        \item \underline{Unit Existence Condition}: For every morphism $f$, there exist units $u_C$ and $u_D$ of $(M, \circ)$ such that $u_C \circ f$ and $f \circ u_D$ are defined.
    \end{enumerate}
    Given any such object-free category $\mathfrak{P}=(M,\circ)$, the corresponding $\amst$ is denoted as $\mstr_{\mathfrak{P}}=(M\times M,\models, \mathcal{P}(M))$, where for all $\Gamma\subseteq M$ and $(m,n)\in M\times M$, $(m,n)\models \Gamma$ iff $m\circ n$ is defined, and $\Gamma=\{m\circ n\}$. 

    Conversely, consider an $\amst$ of the form $\mstr=(M\times M,\models \mathcal{P}(M))$. Let us call an element $u\in M$ an \emph{$\mstr$-identity} if for all $g\in M$ the following statements hold: (a) $(u,g)\models \{g\}$ and $(u,g)\models \{k\}$ for some $k\in M$ implies that $k=g$ and (b) $(g,u)\models \{g\}$ and $(g,u)\models \{k\}$ for some $k\in M$ implies that $k=g$. Finally, suppose that $\models$ satisfies the following properties.
    \begin{enumerate}[label=($\text{\alph*}_{\mstr}$)]
        \item For all $(g,f)\in M\times M$ and $\Gamma\subseteq M$, $(g,f)\models \Gamma$ implies that $\Gamma$ is a singleton set.
        \item For all $\{(h,g),(g,f)\}\subseteq M\times M$, the following statements are equivalent.
        \begin{enumerate}[label=($\text{b}_{\mstr}\arabic*$)]
            \item $(g,f)\models \{x\}$ and $(h,g)\models \{y\}$ for some $\{x,y\}\in M$.
            \item If $(\text{b}_{\mstr}1)$ holds for unique $x$ and $y$, then there exists a unique $z\in M$ such that $(h,x)\models \{z\}$ iff $(y,g)\models \{z\}$.
        \end{enumerate}
       \item For all $f\in M$ there exist $\mstr$-identities $\{u_C,u_D\}\subseteq M$ such that $(u_C,f)\models \{f\}$ and $(f,u_D)\models \{f\}$.
    \end{enumerate}
    We now define a pair $(M,\ast_{\mstr})$ where $M$ is a set, $\ast_{\mstr}$ is a binary relation on $M$ defined as follows.
    $$x\ast_{\mstr}y~\text{iff there exists a unique}~l\in M~\text{such that}~(x,y)\models\{l\}$$This, however, gives rise to a partial binary operation on $M$, say $\odot_{\mstr}$, as explained below. 
    
    For any $\{x,y\}\subseteq M$, we say that $\odot_{\mstr}(x,y)$ is defined if $x\ast_{\mstr}y$. In that case, we choose the unique $l\in M$ such that $(x,y)\models \{l\}$ and let $\odot_{\mstr}(x,y)=l$.
    
    Then, the following statements hold. 
    \begin{enumerate}[label=(\alph*)]
        \item If $(M,\odot_{\mstr})$ satisfies $(\text{b}_{\mstr})$, then it satisfies both $(\text{a}_{\mathfrak{P}})$ and $(\text{b}_{\mathfrak{P}})$.
        \item If $(M,\odot_{\mstr})$ satisfies $(\text{c}_{\mstr})$, then it satisfies $(\text{c}_{\mathfrak{P}})$.
    \end{enumerate}
    Consequently, it follows that $(M,\odot_{\mstr})$ is an object-free category.
\end{ex}

{\hypertarget{sec:nomamst}{\section{Normal \texorpdfstring{$\mathsf{amst}$}{amst} and Compactness}}\label{sec:normamst}}

The concept of compact $\amst$s was introduced in the introductory section. It is, however, not hard to see that not every $\amst$ is compact. 

For example, consider the following $\amst$ $(\mathbb{N},\models,\mathcal{P}(\mathbb{Q}_0))$, where $\mathbb{N}$ denotes the set of natural numbers (including $0$), $\mathbb{Q}_0$ the set of non-negative rational numbers, and for all $X\subseteq \mathbb{Q}_0$ and $n\in \mathbb{N}$, $n\models X$ iff $X$ is bounded below and $n<\inf(X)$. 

Now, let $S=\left\{\frac{1}{n}\mid\,n\in\mathbb{N}\setminus\{0\}\right\}$. Then, $0\models T$, as $0<\inf(T)$ for all finite $T\subseteq S$, i.e., in other words, $S$ is finitely satisfiable. We claim that $S$ is not satisfiable. Suppose the contrary. Then, there exists $n\in \mathbb{N}$ such that $n\models S$, or equivalently, $n<\inf(S)$. Since $\inf(S)\le \dfrac{1}{2}<1$, we must have $n=0$. But, as $0=\inf(S)$, this is a contradiction. Thus, $(\mathbb{N},\models,\mathcal{P}(\mathbb{Q}_0))$ is not compact.

Nevertheless, the following result holds for all $\amst$s.

\hypertarget{thm:finsat=>maxfinsat}{\begin{thm}[\textsc{Extension Lemma}]{\label{thm:finsat=>maxfinsat}}
     Suppose $\mstr=(\mathbf{M},\models,\mathcal{P}(\lang))$ is an $\mathsf{amst}$ and $\Gamma\subseteq\lang$ is a finitely satisfiable set. Then, $\Gamma$ is contained in a maximal finitely satisfiable set (i.e., a set which is finitely satisfiable, but no proper superset of it is so)\footnote{Compare with \cite[Lindenbaum's Lemma]{Goldblatt1984}.}.
\end{thm}}

\begin{proof}
    Let $\mathcal{T}=\{\Sigma\subseteq\lang\mid\,\Gamma\subseteq \Sigma\hbox{ and }\Sigma\hbox{ is finitely satisfiable}\}$. Clearly, $\mathcal{T}\ne\emptyset$ since $\Gamma\in \mathcal{T}$. Let $\mathcal{C}\subseteq \mathcal{T}$ be a chain and $\Delta=\displaystyle\bigcup_{\Sigma\in \mathcal{C}}\Sigma$. Clearly, $\Delta\supseteq \Gamma$. We claim that $\Delta$ is finitely satisfiable as well. 

    Suppose the contrary, i.e., $\Delta$ is not finitely satisfiable. Then, there exists some finite $\Delta_0\subseteq\Delta$ such that $\Delta_0$ is not satisfiable. Therefore, as $\mathcal{C}$ is a chain and $\Delta=\displaystyle\bigcup_{\Sigma\in \mathcal{C}}\Sigma$, there exists $\Sigma\in \mathcal{C}$ such that $\Delta_0\subseteq \Sigma$. Since $\Sigma\in \mathcal{C}$, $\Sigma\in \mathcal{T}$, and hence is finitely satisfiable. Since $\Delta_0$ is a finite subset of $\Sigma$, $\Delta_0$ must be satisfiable --- a contradiction. Hence, $\Delta$ is finitely satisfiable. Thus, every chain in $\mathcal{T}$ has an upper bound. So, by Zorn's lemma, $\mathcal{T}$ has a maximal element.  
\end{proof}

Given an $\mathsf{amst}$ $\mstr=(\mathbf{M},\models,\mathcal{P}(\lang))$, we define a map $\Mod: \mathcal{P}(\lang)\to \mathcal{P}(\mathbf{M})$ as follows. For all $\Gamma\subseteq \lang$,
\[
\Mod(\Gamma)=\{m\in \mathbf{M}\mid\,m\models\Gamma\}.
\]

\begin{defn}[\textsc{Complete Set}]
    Given an $\mathsf{amst}$ $\mstr=(\mathbf{M},\models,\mathcal{P}(\lang))$, a set $\Gamma\subseteq \lang$ is said to be \emph{complete} if, $\Mod(\Gamma)\neq\emptyset$ and for all $\alpha\in \lang$, either $\Mod(\Gamma)\subseteq \Mod(\{\alpha\})$ or $\Mod(\Gamma)\subseteq \mathbf{M}\setminus \Mod(\{\alpha\})$. 
\end{defn}

\hypertarget{rem:comp=>sat}{\begin{rem}{\label{rem:comp=>sat}}
    Suppose $\mstr=(\mathbf{M},\models,\mathcal{P}(\lang))$ is an $\amst$ and $\Gamma\subseteq\lang$ is complete. Then, $\Mod(\Gamma)\neq\emptyset$ and hence, $\Gamma$ is satisfiable. Thus, every complete set is satisfiable.
\end{rem}}

\begin{rem}
  The above definition is a `semantic' abstraction of the following concept. A set of formulas $\Gamma$ in a logic\footnote{A logical structure $(\lang,\deduc)$ is called a \emph{logic} if $\lang$ is an algebra. In that case, the elements of $\lang$ are called formulas.} $(\lang,\deduc)$, that has a unary operator $\neg$ (negation), is \emph{complete} iff there exists $\beta\in\lang$ such that $\Gamma\not\deduc\beta$ (this notion has been formally defined in \hyperlink{def:consistent}{\defref{def:consistent}}), and for any formula $\alpha\in\lang$, either $\Gamma\deduc\alpha$ or $\Gamma\deduc\neg\alpha$. (See, e.g., \cite[p. 84, Definition i]{Mendelson2015}.)
\end{rem}

As mentioned in the introduction, we will discuss the compactness of special kinds of $\amst$s in this article. These are defined as follows.

\hypertarget{defn:normamst}{\begin{defn}[\textsc{Normal Abstract Model Structure}]{\label{def:normamst}}
    An $\mathsf{amst}$ $\mstr=(\mathbf{M},\models,\mathcal{P}(\lang))$ is said to be \emph{normal} if,  for all $m\in\mathbf{M}$ and for all $\Gamma\subseteq \lang$, $m\models\Gamma$ iff $m\models\{\alpha\}$ for all $\alpha\in\Gamma$. 
\end{defn}}

\begin{defn}[\textsc{Logical Structure induced by an $\amst$}]
    Suppose $\mstr=(\mathbf{M},\models,\mathcal{P}(\lang))$ is an $\amst$. Then, the \emph{logical structure induced by $\mstr$}, denoted by $(\lang,\deduc_{\mstr})$, is such that $\deduc_{\mstr}$ is defined as follows. For all $\Gamma\cup\{\alpha\}\subseteq \lang$, 
    \[
    \Gamma\deduc_{\mstr}\alpha\hbox{ iff, for all } m\in \mathbf{M}, \hbox{ if }m\models\Gamma\hbox{ then }m\models\{\alpha\},\hbox{ i.e., }\Mod(\Gamma)\subseteq\Mod(\{\alpha\}).
    \]
\end{defn}

One can show that the logical structures induced by normal $\amst$s are of \emph{Tarski-type} and that every Tarski-type logical structure is induced by some normal $\amst$. Tarski-type logical structures were defined in \cite{RoyBasuChakraborty2023}. We include the definition below for the sake of the reader.

\hypertarget{defn:Tarski}{\begin{defn}[\textsc{Tarski-type Logical structure}]{\label{def:Tarski}}
 A logical structure $(\lang, \deduc)$ is said to be of \emph{Tarski-type} if $\deduc$ satisfies the following properties.
 \begin{enumerate}[label=(\alph*)]
     \item For all $\Gamma\subseteq \lang$ and  $\alpha\in \Gamma$, $\Gamma\deduc\alpha$. (Reflexivity)
     \item For all $\Gamma\cup\Sigma\cup\{\alpha\}\subseteq \lang$, $\Gamma\deduc\alpha$ and $\Gamma\subseteq \Sigma$ implies that $\Sigma\deduc\alpha$. (Monotonicity)
     \item For all $\Gamma\cup\Sigma\cup\{\alpha\}\subseteq \lang$, $\Gamma\deduc\alpha$ if $\Sigma\deduc\alpha$ and for all $\beta\in \Sigma$, $\Gamma\deduc\beta$. (Transitivity) 
 \end{enumerate}   
\end{defn}}

\hypertarget{rem:operator}{\begin{rem}{\label{rem:operator}}
Suppose $(\lang,\deduc)$ is a logical structure and $\Gamma\subseteq\lang$. Let $C_{\deduc}(\Gamma)$ be the set of all elements of $\lang$ that are $\deduc$-related to $\Gamma$, i.e., $C_{\deduc}(\Gamma)=\{\alpha\in \lang\mid\,\Gamma\deduc \alpha\}$. 
Thus, given a relation $\deduc\,\subseteq\mathcal{P}(\lang)\times\lang$, there is a corresponding operator $C_{\deduc}:\mathcal{P}(\lang)\to\mathcal{P}(\lang)$, where, for any $\Gamma\in\mathcal{P}(\lang)$, $C_{\deduc}(\Gamma)$ is as described above.
Conversely, given an operator $C:\mathcal{P}(\lang)\to\mathcal{P}(\lang)$, we can define a relation $\deduc\,\subseteq\mathcal{P}(\lang)\times\lang$ such that $C=C_{\deduc}$ as follows. For all $\Gamma\cup\{\alpha\}\subseteq\lang$, $\Gamma\deduc\alpha$ iff $\alpha\in C(\Gamma)$. We will thus, move freely back and forth between a relation $\deduc$ and its corresponding operator $C_{\deduc}$ and on some occasions, define a $\deduc\,\subseteq\mathcal{P}(\lang)\times\lang$ such that $C_{\deduc}=C$, for some operator $C:\mathcal{P}(\lang)\to\mathcal{P}(\lang)$ in this very sense. As a consequence of this interchangeability between $\deduc$ and $C_{\deduc}$, the above definition of a Tarski-type logical structure can be rephrased in terms of $C_{\deduc}$ as follows.
\begin{enumerate}[label=(\alph*)]
\item For all $\Gamma\subseteq \lang$, $\Gamma\subseteq C_{\deduc}(\Gamma)$. (Reflexivity)
\item For all $\Gamma,\Sigma\subseteq \lang$, if $\Gamma\subseteq \Sigma$ then $C_{\deduc}(\Gamma)\subseteq C_{\deduc}(\Sigma)$. (Monotonicity)
\item For all $\Gamma,\Sigma\subseteq \lang$, if $\Sigma\subseteq {C_{\deduc}}(\Gamma)$ then ${C_{\deduc}}(\Sigma)\subseteq {C_{\deduc}}(\Gamma)$. (Transitivity)
\end{enumerate}\end{rem}}

\hypertarget{defn:consistent}{\begin{defn}{\label{def:consistent}}
    Suppose $(\lang,\deduc)$ is a logical structure and $C_{\deduc}:\pow(\lang)\to\pow(\lang)$ is the operator corresponding to $\deduc$ as described above.
    \begin{enumerate}[label=(\roman*)]
        \item A set $\Gamma\subseteq \lang$ is said to be \emph{$\deduc$-nontrivial} (or \emph{consistent}) if $C_{\deduc}(\Gamma)\ne\lang$, i.e., there exists $\beta\in\lang$ such that $\Gamma\not\deduc\beta$. $\Gamma$ is said to be \emph{$\deduc$-trivial} (or \emph{inconsistent}) otherwise.
        \item A set $\Gamma\subseteq \lang$ is said to be \emph{$\deduc$-closed} if $C_{\deduc}(\Gamma)=\Gamma$, i.e., for all $\beta\in\lang$, $\Gamma\deduc\beta$ iff $\beta\in\Gamma$.
    \end{enumerate}
\end{defn}}

\hypertarget{thm:rep_Tarski}{\begin{thm}[\textsc{Representation for Tarski-type}]{\label{thm:rep_Tarski}}
    \begin{enumerate}[label=(\arabic*)]
        \item \ The logical structure induced by a normal $\amst$ is of Tarski-type.
        \item Given any Tarski-type logical structure $(\lang,\deduc)$, there exists a normal $\mathsf{amst}$ $\mstr$ such that $\lang$ is not satisfiable in $\mstr$ and the logical structure induced by $\mstr$ coincides with $(\lang,\deduc)$.
    \end{enumerate}
\end{thm}}

\begin{proof}
The proof of (1) is a routine verification of the conditions stated in \hyperlink{def:Tarski}{\defref{def:Tarski}}. The main ideas of the proof of (2) can be found in \cite{Beziau1999}. We, however, include the proofs below, for the sake of completeness.

\begin{enumerate}[label=(\arabic*)]
    \item Suppose $\mstr=(\mathbf{M},\models,\pow(\lang))$ is a normal $\amst$ and $(\lang,\deduc_\mstr)$ the logical structure induced by $\mstr$. 
    
    Let $\alpha\in\Gamma\subseteq\lang$ and $m\in\mathbf{M}$ such that $m\models\Gamma$. Since $\mstr$ is normal, $m\models\{\alpha\}$. Thus, $\Gamma\deduc_\mstr\alpha$ for all $\alpha\in\Gamma$.  Hence, $\deduc_\mstr$ satisfies reflexivity.

    Let $\Gamma\cup\Sigma\cup\{\alpha\}\subseteq\lang$, $\Gamma\subseteq\Sigma$ and $\Gamma\deduc_\mstr\alpha$. Moreover, let $m\in\mathbf{M}$ such that $m\models\Sigma$. Since $\mstr$ is normal, this implies that $m\models\{\beta\}$ for all $\beta\in\Sigma$. So, $m\models\{\beta\}$ for all $\beta\in\Gamma$, as $\Gamma\subseteq\Sigma$. Again, by normality, this implies that $m\models\Gamma$. Then, as $\Gamma\deduc_\mstr\alpha$, $m\models\{\alpha\}$. Thus, $m\models\Sigma$ implies $m\models\{\alpha\}$. So, $\Sigma\deduc_\mstr\alpha$. Hence, $\deduc_\mstr$ satisfies monotonicity.

    Let $\Gamma\cup\Sigma\cup\{\alpha\}\subseteq\lang$ such that $\Sigma\deduc_\mstr\alpha$ and $\Gamma\deduc_\mstr\beta$ for all $\beta\in\Sigma$. Moreover, let $m\in\mathbf{M}$ such that $m\models\Gamma$. Since, $\Gamma\deduc_\mstr\beta$ for all $\beta\in\Sigma$, this implies that $m\models\{\beta\}$ for all $\beta\in\Sigma$. Then, as $\mstr$ is normal, $m\models\Sigma$. Now, since $\Sigma\deduc_\mstr\alpha$, we have $m\models\{\alpha\}$. Thus, $m\models\Gamma$ implies $m\models\{\alpha\}$. So, $\Gamma\deduc_\mstr\alpha$. Hence, $\deduc_\mstr$ satisfies transitivity.

    This proves that $(\lang,\deduc_\mstr)$ is a Tarski-type logical structure.
    
    \item Let $(\lang,\deduc)$ be a Tarski-type logical structure, and $C_{\deduc}:\mathcal{P}(\lang)\to \mathcal{P}(\lang)$ the operator corresponding to $\deduc$ (as described in \hyperlink{rem:operator}{\remref{rem:operator}}). 
    
    For each $\Sigma\subseteq \lang$, let $\chi_\Sigma:\lang\to\{0,1\}$ be the characteristic function of $\Sigma$, i.e., 
    \[
    \chi_\Sigma(\alpha)=\begin{cases}
    1&\hbox{if }\alpha\in\Sigma\\
    0&\hbox{otherwise}.
    \end{cases}
    \]
We now define an abstract model structure $\mathfrak{B}=(\mathbf{B},\models,\mathcal{P}(\lang))$, where
\begin{itemize}
    \item $\mathbf{B}=\{\chi_{C_{\deduc}(\Sigma)}\mid\,\Sigma\hbox{ is $\deduc$-nontrivial}\}$, and 
    \item for all $\Gamma\cup\Sigma\subseteq \lang$, with $\Sigma$ $\deduc$-nontrivial, $\chi_{C_{\deduc}(\Sigma)}\models\Gamma$ iff $\chi_{C_{\deduc}(\Sigma)}(\alpha)=1$ for all $\alpha\in \Gamma$.
\end{itemize}
It follows from the above definition of $\mathfrak{B}$ that, for any $\chi_{C_{\deduc}(\Sigma)}\in\mathbf{B}$ and any $\Gamma\subseteq\lang$,
\[
\begin{array}{llll}
     \chi_{C_{\deduc}(\Sigma)}\models\Gamma&\hbox{iff}&\chi_{C_{\deduc}(\Sigma)}(\alpha)=1&\hbox{for all }\alpha\in                                  \Gamma\\
     &\hbox{iff}&\chi_{C_{\deduc}(\Sigma)}\models \{\alpha\}&\hbox{for all }\alpha\in\Gamma
\end{array}
\]
This shows that $\mathfrak{B}$ is a normal $\amst$.

We now show that the logical structure $(\lang,\deduc_\mathfrak{B})$, induced by $\mathfrak{B}$, is such that $\deduc\,=\,\deduc_{\mathfrak{B}}$. 

Suppose $\Gamma\cup\{\alpha\}\subseteq \lang$ such that $\Gamma\deduc\alpha$ but $\Gamma\not\deduc_{\mathfrak{B}}\alpha$. Since $\Gamma\not\deduc_{\mathfrak{B}}\alpha$, there exists a $\deduc$-nontrivial $\Sigma\subseteq \lang$ such that $\chi_{C_{\deduc}(\Sigma)}\models \Gamma$ but $\chi_{C_{\deduc}(\Sigma)}\not\models \{\alpha\}$. Now, $\chi_{C_{\deduc}(\Sigma)}\models \Gamma$ implies that $\chi_{C_{\deduc}(\Sigma)}(\beta)=1$ for all $\beta\in\Gamma$, which means that $\Gamma\subseteq C_{\deduc}(\Sigma)$. On the other hand, by similar reasoning, $\chi_{C_{\deduc}(\Sigma)}\not\models \{\alpha\}$ implies that $\alpha\notin C_{\deduc}(\Sigma)$. However, since $\Gamma\subseteq C_{\deduc}(\Sigma)$ and $\alpha\in C_{\deduc}(\Gamma)$, by transitivity of $C_{\deduc}$, $\alpha\in C_{\deduc}(\Sigma)$. This is a contradiction. Thus, $\Gamma\deduc\alpha$ implies $\Gamma\deduc_\mathfrak{B}\alpha$, i.e.,  $\deduc\,\subseteq\,\deduc_{\mathfrak{B}}$.

Next, suppose $\Gamma\cup\{\alpha\}\subseteq\lang$ such that $\Gamma\deduc_{\mathfrak{B}}\alpha$ but $\Gamma\not\deduc\alpha$, i.e., $\alpha\notin C_{\deduc}(\Gamma)$. Thus, $\Gamma$ is $\deduc$-nontrivial and hence, $\chi_{C_{\deduc}(\Gamma)}\in \mathbf{B}$. Now, by reflexivity, $\chi_{C_{\deduc}(\Gamma)}\models \Gamma$. However, as  $\alpha\notin C_{\deduc}(\Gamma)$, $\chi_{C_{\deduc}(\Gamma)}(\alpha)\ne 1$, which implies that $\chi_{C_{\deduc}(\Gamma)}\not\models \{\alpha\}$. So, $\Gamma\not\deduc_{\mathfrak{B}}\alpha$. This is a contradiction. Thus, $\Gamma\deduc_{\mathfrak{B}}\alpha$ implies $\Gamma\deduc\alpha$, i.e.,  $\deduc_{\mathfrak{B}}\,\subseteq\,\deduc$.

Hence, $\deduc\,=\,\deduc_\mathfrak{B}$.

Finally, to see that $\lang$ is not satisfiable in $\mathfrak{B}$, we first note that, since $(\lang,\deduc)$ is of Tarski-type, $\lang$ is not $\deduc$-nontrivial. Now, suppose $\Sigma\subseteq\lang$ is $\deduc$-nontrivial such that $\chi_{C_{\deduc}(\Sigma)}\models\lang$. Then, $\chi_{C_{\deduc}(\Sigma)}(\alpha)=1$, i.e., $\alpha\in\Sigma$ for all $\alpha\in\lang$. This implies that $\Sigma=\lang$, which contradicts the assumption that $\Sigma$ is $\deduc$-nontrivial. Hence, $\lang$ is not satisfiable in $\mathfrak{B}$.
\end{enumerate} 
\end{proof}

\hypertarget{thm:nsat=>triv}{\begin{thm}{\label{thm:nsat=>triv}}
Suppose $\mstr=(\mathbf{M},\models,\pow(\lang))$ is an $\amst$ and $(\lang,\deduc_\mstr)$ is the logical structure induced by $\mstr$. If $\Gamma\subseteq\lang$ is not satisfiable, then it is $\deduc_\mstr$-trivial. The converse holds if $\mstr$ is normal and $\lang$ is not satisfiable.
\end{thm}}

\begin{proof}
    Let $\Gamma\subseteq\lang$ be not satisfiable. Suppose $\Gamma$ is $\deduc_\mstr$-nontrivial. Then, there exists $\alpha\in\lang$ such that $\Gamma\not\deduc_\mstr\alpha$. So, there exists $m\in\mathbf{M}$ such that $m\models\Gamma$ but $m\not\models\{\alpha\}$. Now, as $m\models\Gamma$, $\Gamma$ is satisfiable. This is a contradiction. Thus, $\Gamma$ must be $\deduc_\mstr$-trivial.

    Now, suppose $\mstr$ is normal and $\lang$ is not satisfiable. Let $\Gamma\subseteq\lang$ be $\deduc_\mstr$-trivial. Suppose $\Gamma$ is satisfiable. Then, there exists $m\in\mathbf{M}$ such that $m\models\Gamma$. Now, since $\Gamma$ is $\deduc_\mstr$-trivial, $\Gamma\deduc_\mstr\alpha$ for all $\alpha\in\lang$. Thus, $m\models\{\alpha\}$ for all $\alpha\in\lang$. Then, by normality of $\mstr$, $m\models\lang$. This, however, contradicts the assumption that $\lang$ is not satisfiable. Hence, $\Gamma$ is not satisfiable.
\end{proof}

\hypertarget{thm:cumulative}{\begin{thm}[\textsc{Properties of} $\Mod$]{\label{thm:cumulative}}
Suppose $\mstr=(\mathbf{M},\models,\mathcal{P}(\lang))$ is a normal $\mathsf{amst}$. Then, the following statements hold.
\begin{enumerate}[label=(\arabic*)]
    \hypertarget{thm:cumulative(1)}{\item For all $\Gamma\subseteq\lang$, $\Mod(\Gamma)=\displaystyle\bigcap_{\alpha\in \Gamma}\Mod(\{\alpha\})$. \label{thm:cumulative(1)}}
    \hypertarget{thm:cumulative(2)}{\item For all $\Gamma\subseteq\Sigma\subseteq\lang$, $\Mod(\Sigma)\subseteq \Mod(\Gamma)$. \label{thm:cumulative(2)}}
    \hypertarget{thm:cumulative(3)}{\item For any family $(\Sigma_i)_{i\in I}$ of subsets of $\lang$,  
    \[
    \Mod\left(\displaystyle\bigcup_{i\in I}\Sigma_i\right)=\displaystyle\bigcap_{i\in I}\Mod(\Sigma_i)\quad\hbox{ and }\quad\displaystyle\bigcup_{i\in I}\Mod(\Sigma_i)\subseteq\Mod\left(\displaystyle\bigcap_{i\in I}\Sigma_i\right).
    \]\label{thm:cumulative(3)}}
\end{enumerate}
\end{thm}}

\begin{proof}
\begin{enumerate}[label=(\arabic*)]
    \item Let $\Gamma\subseteq \lang$. 

If $\Gamma=\emptyset$, then $\displaystyle\bigcap_{\alpha\in \Gamma}\Mod(\{\alpha\})=\mathbf{M}$. We claim that $\Mod(\emptyset)=\mathbf{M}$. 

Suppose the contrary, i.e., $\Mod(\emptyset)\ne\mathbf{M}$. Then, there exists $m\in \mathbf{M}$ such that $m\not\models\emptyset$. Now, as $\mstr$ is normal, this implies that there exists $\alpha\in \emptyset$ such that $m\not\models\{\alpha\}$. This is an impossibility and so, $\Mod(\emptyset)=\mathbf{M}$. Thus (1) holds for $\Gamma=\emptyset$. 

Now, suppose $\Gamma\ne \emptyset$. Now, for any $m\in\mathbf{M}$,
\[
\begin{array}{llll}
     m\in\Mod(\Gamma)&\hbox{iff}&m\models\Gamma\\
     &\hbox{iff}&m\models\{\alpha\}&\hbox{ for all $\alpha\in\Gamma$, since $\mstr$ is normal}\\
     &\hbox{iff}&m\in\Mod(\{\alpha\})&\hbox{ for all }\alpha\in\Gamma\\
     &\hbox{iff}&m\in\displaystyle\bigcap_{\alpha\in \Gamma}\Mod(\{\alpha\})
\end{array}
\]
Thus, $\Mod(\Gamma)=\displaystyle\bigcap_{\alpha\in \Gamma}\Mod(\{\alpha\})$ for all $\Gamma\subseteq\lang$.

\item Let $\Gamma\subseteq\Sigma\subseteq\lang$. Then,
\[
\begin{array}{lcll}
     \Mod(\Sigma)&=&\displaystyle\bigcap_{\alpha\in \Sigma}\Mod(\{\alpha\}),&\hbox{by (1)}\\
     &\subseteq&\displaystyle\bigcap_{\alpha\in \Gamma}\Mod(\{\alpha\}),&\hbox{since }\Gamma\subseteq\Sigma\\
     &=&\Mod(\Gamma),&\hbox{again by (1)}.
\end{array}
\]

\item It follows from (2) that $\Mod\left(\displaystyle\bigcup_{i\in I}\Sigma_i\right)\subseteq\displaystyle\bigcap_{i\in I}\Mod(\Sigma_i)$, since $\Sigma_i\subseteq\displaystyle\bigcup_{i\in I}\Sigma_i$, and hence, $\Mod\left(\displaystyle\bigcup_{i\in I}\Sigma_i\right)\subseteq\Mod(\Sigma_i)$, for each $i\in I$. 

Now, let $m\in\displaystyle\bigcap_{i\in I}\Mod(\Sigma_i)$. Then, $m\in\Mod(\Sigma_i)$, i.e., $m\models\Sigma_i$, for all $i\in I$. Since $\mstr$ is normal, this implies that $m\models\displaystyle\bigcup_{i\in I}\Sigma_i$. So, $m\in\Mod\left(\displaystyle\bigcup_{i\in I}\Sigma_i\right)$. Thus, $\displaystyle\bigcap_{i\in I}\Mod(\Sigma_i)\subseteq\Mod\left(\displaystyle\bigcup_{i\in I}\Sigma_i\right)$. Hence, $\Mod\left(\displaystyle\bigcup_{i\in I}\Sigma_i\right)=\displaystyle\bigcap_{i\in I}\Mod(\Sigma_i)$.

Again, by (2), $\displaystyle\bigcup\Mod(\Sigma_i)\subseteq\Mod\left(\bigcap_{i\in I}\Sigma_i\right)$, since $\displaystyle\bigcap_{i\in I}\Sigma_i\subseteq\Sigma_i$, and hence, $\Mod(\Sigma_i)\subseteq\Mod\left(\displaystyle\bigcap_{i\in I}\Sigma_i\right)$, for each $i\in I$.
\end{enumerate}
\end{proof}

{\hypertarget{subsec:com-compl}{\subsection{Compactness and Maximal Satisfiablity}}\label{subsec:com-compl}}

A common proof of the compactness theorem for classical propositional logic (or even for first-order logic) rests on the crucial step of extending a finitely satisfiable set to a nontrivial maximal finitely satisfiable one. We show in this section that this process can be replicated for normal $\amst$s.

\hypertarget{thm:maxfinsat=>closed}{\begin{thm}{\label{thm:maxfinsat=>closed}}
Suppose $\mstr=(\mathbf{M},\models,\mathcal{P}(\lang))$ is a normal $\mathsf{amst}$ and $(\lang,\deduc_{\mstr})$ the logical structure induced by $\mstr$. Let $\Gamma\subseteq \lang$ be maximal finitely satisfiable. If $\Gamma$ is satisfiable, then it is $\deduc_{\mstr}$-closed.
\end{thm}}

\begin{proof}
Suppose $\Gamma$ is satisfiable, i.e., there exists $m\in \mathbf{M}$ such that $m\models \Gamma$, but $\Gamma$ is not $\deduc_{\mstr}$-closed. Since $\mstr$ is normal, by \hyperlink{thm:rep_Tarski}{\thmref{thm:rep_Tarski}}, $(\lang,\deduc_\mstr)$ is of Tarski-type. Thus, $\Gamma\deduc_\mstr\alpha$ for all $\alpha\in\Gamma$. Then, as $\Gamma$ is not $\deduc_\mstr$-closed, there exists $\alpha\in \lang$ such that $\Gamma\deduc_{\mstr}\alpha$ but $\alpha\notin \Gamma$. Now, since $\Gamma$ is maximal finitely satisfiable and $\Gamma\subsetneq \Gamma\cup\{\alpha\}$, $\Gamma\cup\{\alpha\}$ is not finitely satisfiable. So, there exists a finite $\Gamma_0\subseteq \Gamma\cup\{\alpha\}$, which is not satisfiable. 

If $\Gamma_0\subseteq \Gamma$, then this would imply that $\Gamma$ is not finitely satisfiable, contrary to our hypothesis. Thus, $\alpha\in \Gamma_0$ and hence, $\Gamma_0=\Delta_0\cup\{\alpha\}$ for some finite $\Delta_0\subseteq \Gamma$. 
    
Since $\mstr$ is normal, $m\models\Gamma$ and $\Delta_0\subseteq\Gamma$, $m\models \Delta_0$ as well. Then, as $\Delta_0\cup\{\alpha\}$ is not satisfiable, $m\not\models\{\alpha\}$. This implies that $\Gamma\not\deduc_{\mstr}\alpha$, which is a contradiction. So, $\alpha\in \Gamma$. Hence, $\Gamma$ is $\deduc_{\mstr}$-closed. 
\end{proof}

\hypertarget{cor:nontriv=>sat}{\begin{cor}{\label{cor:nontriv=>sat}}
Suppose $\mstr=(\mathbf{M},\models,\mathcal{P}(\lang))$ is a normal $\mathsf{amst}$ and $(\lang,\deduc_{\mstr})$ the logical structure induced by $\mstr$. Let $\Gamma\subsetneq \lang$ be maximal finitely satisfiable. Then $\Gamma$ is satisfiable iff it is $\deduc_{\mstr}$-nontrivial.
\end{cor}}

\begin{proof}
Suppose $\Gamma$ is satisfiable. Then, by \hyperlink{thm:maxfinsat=>closed}{\thmref{thm:maxfinsat=>closed}}, $\Gamma$ is $\deduc_{\mstr}$-closed. Let $\alpha\in \lang\setminus\Gamma$ (such an $\alpha$ exists as $\Gamma\subsetneq \lang$). Now, since $\Gamma$ is $\deduc_{\mstr}$-closed, $\Gamma\not\deduc_{\mstr}\alpha$. Thus, $\Gamma$ is $\deduc_{\mstr}$-nontrivial.

Conversely, suppose $\Gamma$ is $\deduc_{\mstr}$-nontrivial. Then, by \hyperlink{thm:nsat=>triv}{\thmref{thm:nsat=>triv}}, $\Gamma$ is satisfiable.
\end{proof}

Given an $\amst$ $\mstr=(\mathbf{M},\models,\mathcal{P}(\lang))$, we define a map $\Th:\mathcal{P}(\mathbf{M})\to \mathcal{P}(\lang)$ as follows. For any $X\subseteq\mathbf{M}$,
\[
\Th(X)=\{\alpha\in \lang\mid\,m\models \{\alpha\}\hbox{ for all }m\in X\}
\]

We now list some properties of the $\Th$-operator in the following theorem.

\hypertarget{thm:cumulative_th}{\begin{thm}[\textsc{Properties of}~$\Th$]{\label{thm:cumulative_th}}
    Suppose $\mstr=(\mathbf{M},\models,\mathcal{P}(\lang))$ is an $\amst$. Let $\Th:\pow(\mathbf{M})\to\pow(\lang)$ be as defined above. Then, the following statements hold.
    \begin{enumerate}[label=(\arabic*)]
        \hypertarget{thm:cumulative_th(1)}{\item For all $X\subseteq\mathbf{M}$, $\Th(X)=\displaystyle\bigcap_{m\in X}\Th(\{m\})$.  \label{thm:cumulative_th(1)}}
        \hypertarget{thm:cumulative_th(2)}{\item For all $X\subseteq Y\subseteq\mathbf{M}$, $\Th(Y)\subseteq \Th(X)$.\label{thm:cumulative_th(2)}}
        \hypertarget{thm:cumulative_th(3)}{\item If $X=\emptyset$, then $\Th(X)=\lang$. The converse holds if $\lang$ is not finitely satisfiable.\label{thm:cumulative_th(3)}}
        \hypertarget{thm:cumulative_th(4)}{\item For any family $(X_i)_{i\in I}$ of subsets of $\mathbf{M}$,
        \[
        \Th\left(\bigcup_{i\in I}X_i\right)=\bigcap_{i\in I}\Th(X_i)\quad\hbox{ and }\quad\bigcup_{i\in I}\Th(X_i)\subseteq \Th\left(\bigcap_{i\in I}X_i\right).
        \]\label{thm:cumulative_th(4)}}
        \hypertarget{thm:cumulative_th(5)}{\item Suppose $\mstr$ is normal and let $(\lang,\deduc_\mstr)$ be the logical structure induced by $\mstr$. Then, for all $\Sigma\subseteq \lang$, $\Sigma\subseteq\Th(\Mod(\Sigma))$. Moreover, for all $\Sigma\subsetneq \lang$, $\Th(\Mod(\Sigma))=\Sigma$ iff $\Sigma$ is satisfiable and $\deduc_{\mstr}$-closed.\label{thm:cumulative_th(5)}}
    \end{enumerate}
\end{thm}}
\begin{proof}
    \begin{enumerate}[label=(\arabic*)]
        \item Let $X\subseteq\mathbf{M}$.

        If $X=\emptyset$, then $\displaystyle\bigcap_{m\in X}\Th(\{m\})=\lang$. Suppose $\Th(\emptyset)\neq\lang$. Then, there exists $\alpha\in\lang$ such that $m\not\models\{\alpha\}$ for some $m\in\emptyset$. This is impossible. Thus, $\Th(\emptyset)=\lang$ and hence, in this case, $\Th(X)=\displaystyle\bigcap_{m\in X}\Th(\{m\})$.

        Now, suppose $X\neq\emptyset$. Then, for any $\alpha\in\lang$,
        \[
        \begin{array}{llll}
             \alpha\in\Th(X)&\hbox{iff}&m\models\{\alpha\}&\hbox{for all }m\in X\\
             &\hbox{iff}&\alpha\in\Th(\{m\})&\hbox{for all }m\in X\\
             &\hbox{iff}&\alpha\in\displaystyle\bigcap_{m\in X}\Th(\{m\})
        \end{array}
        \]
        Thus, $\Th(X)=\displaystyle\bigcap_{m\in X}\Th(\{m\})$ for all $X\subseteq\mathbf{M}$.

        \item Let $X\subseteq Y\subseteq\mathbf{M}$. Then,
        \[
        \begin{array}{lcll}
            \Th(Y)&=&\displaystyle\bigcap_{m\in Y}\Th(\{m\}),&\hbox{by (1)}\\
            &\subseteq&\displaystyle\bigcap_{m\in X}\Th(\{m\}),&\hbox{since }X\subseteq Y\\
            &=&\Th(X),&\hbox{again by (1)}.
        \end{array}
        \]

        \item We have shown above, in the proof of (1), that if $X=\emptyset$, then $\Th(X)=\lang$. 
        
        Now, for the converse, suppose $\lang$ is not finitely satisfiable and $\Th(X)=\lang$. If possible, let $m\in X$. Then, as $\{m\}\subseteq X$, by (2), we have $\Th(X)\subseteq \Th(\{m\})$. This implies that $\Th(\{m\})=\lang$. So, $m\models \{\alpha\}$ for all $\alpha\in \lang$. This, however, contradicts our assumption that $\lang$ is not finitely satisfiable. Hence, $X=\emptyset$.

        \item It follows from (2) that $\Th\left(\displaystyle\bigcup_{i\in I}X_i\right)\subseteq\displaystyle \bigcap_{i\in I} \Th(X_i)$, since $X_i\subseteq\displaystyle \bigcup_{i\in I}X_i$. So, $\Th\left(\displaystyle\bigcup_{i\in I}X_i\right)\subseteq \Th(X_i)$, for each $i\in I$.

        For the reverse inclusion, let $\alpha\in \lang$ such that $\alpha\notin \Th\left(\displaystyle\bigcup_{i\in I}X_i\right)$. Then, there exists $m\in\displaystyle\bigcup_{i\in I}X_i$ such that $m\not\models\{\alpha\}$. Since $m\in\displaystyle\bigcup_{i\in I}X_i$, $m\in X_i$ for some $i\in I$. So, $\alpha\notin\Th(X_i)$ and hence, $\alpha\notin\displaystyle\bigcap_{i\in I}\Th(X_i)$.
        
        Thus, $\displaystyle \bigcap_{i\in I} \Th(X_i)\subseteq  \Th\left(\displaystyle\bigcup_{i\in I}X_i\right)$ as well. Hence, $\Th\left(\displaystyle\bigcup_{i\in I}X_i\right)=\displaystyle \bigcap_{i\in I} \Th(X_i)$.
        
        Again, by (2), $\displaystyle\bigcup_{i\in I}\Th(X_i)\subseteq \Th\left(\bigcap_{i\in I}X_i\right)$, since $\displaystyle\bigcap_{i\in I}X_i\subseteq X_i$, and hence, $\Th(X_i)\subseteq\Th\left(\displaystyle\bigcap_{i\in I}X_i\right)$, for each $i\in I$.

    \item Let $\Sigma\subseteq \lang$ and $\alpha\in\lang$ such that $\alpha\notin\Th(\Mod(\Sigma))$. Then, $m\not\models \{\alpha\}$ for some $m\in\Mod(\Sigma)$. Thus, $m\models\Sigma$ but $m\not\models\{\alpha\}$. Since $\mstr$ is normal, this implies that $\alpha\notin\Sigma$. 
    
    Thus, $\Sigma\subseteq \Th(\Mod(\Sigma))$ for all $\Sigma\subseteq \lang$.

    Now, let $\Sigma\subsetneq \lang$ such that $\Th(\Mod(\Sigma))=\Sigma$ and $\alpha\in\lang\setminus\Sigma$. Then, $\alpha\notin \Th(\Mod(\Sigma))$. Hence, there exists $m\in \Mod(\Sigma)$ such that $m\not\models \{\alpha\}$. Since $m\in\Mod(\Sigma)$, $\Sigma$ is satisfiable. 
    
    Since $\mstr$ is normal, $(\lang,\deduc_\mstr)$ is of Tarski-type, by \hyperlink{thm:rep_Tarski}{\thmref{thm:rep_Tarski}}. Thus, $\Sigma\deduc_\mstr\beta$ for all $\beta\in\Sigma$. Now, let $\beta\in \lang$ such that $\Sigma\deduc_{\mstr}\beta$. So, $\Mod(\Sigma)\subseteq \Mod(\{\beta\})$. Then, by (2), $\Th(\Mod(\{\beta\})\subseteq \Th(\Mod(\Sigma))$. Now, by the above arguments, $\{\beta\}\subseteq\Th(\Mod(\{\beta\})$. So, $\{\beta\}\subseteq\Th(\Mod(\Sigma))=\Sigma$. Thus, $\beta\in\Sigma$. Hence, $\Sigma\deduc_\mstr\beta$ iff $\beta\in\Sigma$, i.e., $\Sigma$ is $\deduc_\mstr$-closed.
    
    For the converse, let $\Sigma\subsetneq \lang$ such that $\Sigma$ is satisfiable and $\deduc_{\mstr}$-closed. Suppose $\Sigma\subsetneq \Th(\Mod(\Sigma))$. Then, there exists $\alpha\in \Th(\Mod(\Sigma))$ such that $\alpha\notin \Sigma$. Now, since $\Sigma$ is satisfiable, $\Mod(\Sigma)\neq\emptyset$. Let $m\in \Mod(\Sigma)$. Then, as $\alpha\in \Th(\Mod(\Sigma))$, $m\models \{\alpha\}$, i.e., $m\in \Mod(\{\alpha\})$. Thus, $\Mod(\Sigma)\subseteq\Mod(\{\alpha\})$ and so, $\Sigma\deduc_{\mstr}\alpha$. Since $\Sigma$ is $\deduc_{\mstr}$-closed, this implies that $\alpha\in \Sigma$. This is a contradiction. Thus, $\Th(\Mod(\Sigma))\subseteq\Sigma$. Hence, $\Th(\Mod(\Sigma))=\Sigma$.
    \end{enumerate}
\end{proof}

We now recall some concepts from the theory of partial orders essential for the following theorem.   

A \emph{partially ordered set} or \emph{poset} is a pair $(P,\le)$, where $P$ is a set and $\le\,\subseteq P\times P$ such that for all $x,y\in P$, the following properties are satisfied.
\begin{enumerate}[label=(\alph*)]
    \item $(x,x)\in\,\le$.
    \item If $(x,y)\in\,\le$ and $(y,x)\in\,\le$, then $x=y$.
    \item If $(x,y)\in\,\le$ and $(y,z)\in\,\le$, then $(x,z)\in\,\le$.
\end{enumerate}
For any $x,y\in P$, we usually write $x\le y$ instead of $(x,y)\in\,\le$. 

A poset $(P,\le)$ is said to be \emph{$\le$-directed} if for all $x,y\in P$, there exists $z\in P$ such that $x\le z$ and $y\le z$. 

Suppose $A$ is any set and $X\subseteq\pow(A)$. Then, $(X,\subseteq)$ and $(X,\supseteq)$ are posets.

Suppose $(P,\le_p)$ and $(Q,\le_q)$ are posets. A function $f:P\to Q$ is said to be \emph{order-preserving} (respectively, \emph{order-reversing}) if, for all $x,y\in P$, $x\le_p y$ implies that $f(x)\le_q f(y)$ (respectively, $f(y)\le_q f(x)$). 

Given a set $A$, we denote the set of all finite subsets of $A$ by $\Finset(A)$.

\hypertarget{thm:compact_normal_ams(I)}{\begin{thm}[\textsc{Compactness for Normal $\mathsf{amst}$ (I)}]{\label{thm:compact_normal_ams(I)}}
Suppose $\mstr=(\mathbf{M},\models,\mathcal{P}(\lang))$ is a normal $\mathsf{amst}$ such that $\lang$ is not finitely satisfiable. Let $(\lang,\deduc_{\mstr})$ be the logical structure induced by $\mstr$. Then, the following statements are equivalent.
\begin{enumerate}[label=(\arabic*)]
    \item $\mstr$ is compact. 
    \item Every finitely satisfiable set is contained in a $\deduc_{\mstr}$-nontrivial maximal finitely satisfiable set. 
    \item Every finitely satisfiable set is contained in a maximal satisfiable set. 
    \item Every finitely satisfiable set is contained in a complete set. 
    \hypertarget{thm:compact_normal_ams(I)(5)}{\item Every $\deduc_{\mstr}$-trivial set has a finite $\deduc_{\mstr}$-trivial subset.\label{thm:compact_normal_ams(I)(5)}}   
    \hypertarget{thm:compact_normal_ams(I)(7)}{\item Suppose $f:I\to \mathcal{P}(\lang)$ such that $I\ne\emptyset$ and the poset $(f(I),\subseteq)$ is $\subseteq$-directed. If $f(i)$ is satisfiable for all $i\in I$, then so is $\displaystyle\bigcup_{i\in I}f(i)$.\label{thm:compact_normal_ams(I)(7)}}  
    \item  Suppose $\Sigma\subseteq \lang$ and $f:\Finset(\Sigma)\to \mathcal{P}(\lang)$ is an order-preserving map between the posets $(\Finset(\Sigma),\subseteq)$ and $(\pow(\lang),\subseteq)$. If $f(\Gamma)$ is satisfiable for all $\Gamma\in \Finset(\Sigma)$, then so is $\displaystyle\bigcup_{\Gamma\in \Finset(\Sigma)} f(\Gamma)$.
    \hypertarget{thm:compact_normal_ams(I)(9)}{\item Suppose $f:I\to \mathcal{P}(\mathbf{M})$ such that $I\ne\emptyset$ and the poset $(f(I),\supseteq)$ is $\supseteq$-directed. If $f(i)\ne\emptyset$ for all $i\in I$, then $\displaystyle\bigcup_{i\in I}\Th(f(i))$ is satisfiable.\label{thm:compact_normal_ams(I)(9)}}
    \item Suppose $\Sigma\subseteq \lang$ is finitely satisfiable and $f:\Finset(\Sigma)\to \mathcal{P}(\mathbf{M})$ is an order-reversing map between the posets $(\Finset(\Sigma),\subseteq)$ and $(\pow(\mathbf{M}),\subseteq)$. If $f(\Gamma)\ne\emptyset$ for all $\Gamma\in \Finset(\Sigma)$, then $\displaystyle\bigcup_{\Gamma\in \Finset(\Sigma)} \Th(f(\Gamma))$ is satisfiable.    
\end{enumerate}
\end{thm}}

\begin{proof}
\underline{(1) $\implies$ (2)}: Let $\Gamma\subseteq\lang$ be a finitely satisfiable set. Then, by \hyperlink{thm:finsat=>maxfinsat}{\thmref{thm:finsat=>maxfinsat}}, $\Gamma$ is contained in a maximal finitely satisfiable set, say $\Sigma$. Since $\lang$ is not finitely satisfiable, $\Sigma\subsetneq \lang$. Now, as $\Sigma$ is finitely satisfiable, by (1), $\Sigma$ is satisfiable. So, by \hyperlink{cor:nontriv=>sat}{\corref{cor:nontriv=>sat}}, $\Sigma$ is $\deduc_{\mstr}$ nontrivial. Thus, $\Sigma$ is a $\deduc_{\mstr}$-nontrivial maximal finitely satisfiable set containing $\Gamma$.

\underline{(2) $\implies$ (1)}: Let $\Gamma\subseteq \lang$ be finitely satisfiable. Then, by (2), $\Gamma$ is contained in a $\deduc_{\mstr}$-nontrivial maximal finitely satisfiable set $\Sigma$. Since $\mstr$ is normal, by \hyperlink{thm:rep_Tarski}{\thmref{thm:rep_Tarski}}, $(\lang,\deduc_{\mstr})$ is of Tarski-type. So, by reflexivity of $\deduc_\mstr$, as $\Sigma$ is nontrivial, $\Sigma\subsetneq\lang$. Then, by \hyperlink{cor:nontriv=>sat}{\corref{cor:nontriv=>sat}}, $\Sigma$ is satisfiable. Since $\Gamma\subseteq \Sigma$, by normality of $\mstr$, $\Gamma$ is satisfiable as well.  Thus, $\mstr$ is compact.

\underline{(1) $\implies$ (3)}: Let $\Gamma\subseteq\lang$ be finitely satisfiable. Then, by \hyperlink{thm:finsat=>maxfinsat}{\thmref{thm:finsat=>maxfinsat}}, there exists a maximal finitely satisfiable $\Sigma$ such that $\Gamma\subseteq\Sigma$. Now, as $\Sigma$ is finitely satisfiable, by (1), $\Sigma$ is satisfiable. If possible, let $\Delta\supsetneq \Sigma$ be a satisfiable set. Then, again by (1), $\Delta$ is finitely satisfiable as well. This, however, contradicts the fact that $\Sigma$ is maximal finitely satisfiable. Hence, $\Sigma$ is a maximal satisfiable set containing $\Gamma$.

\underline{(3) $\implies$ (4)}: Let $\Gamma\subseteq\lang$ be finitely satisfiable. Then, by (3), there exists a maximal satisfiable set $\Sigma$ such that $\Gamma\subseteq\Sigma$. We claim that $\Sigma$ is complete. 

Since $\Sigma$ is satisfiable, $\Mod(\Sigma)\neq\emptyset$. Let $\alpha\in \lang$. If $\alpha\in \Sigma$, then by \hyperlink{thm:cumulative(1)}{\thmref{thm:cumulative}(1)}, $\Mod(\Sigma)\subseteq\displaystyle\bigcap_{\beta\in\Sigma}\Mod(\{\beta\})\subseteq\Mod(\{\alpha\})$. On the other hand, if $\alpha\notin \Sigma$, then as $\Sigma$ is maximal satisfiable, $\Sigma\cup\{\alpha\}$ is not satisfiable, i.e., $\Mod(\Sigma\cup\{\alpha\})=\emptyset$. Now, by \hyperlink{thm:cumulative(3)}{\thmref{thm:cumulative}(3)}, $\Mod(\Sigma\cup\{\alpha\})=\Mod(\Sigma)\cap\Mod(\{\alpha\})$. So, $\Mod(\Sigma)\cap\Mod(\{\alpha\})=\emptyset$ as well. This implies that $\Mod(\Sigma)\subseteq \mathbf{M}\setminus \Mod(\{\alpha\})$. Thus, for all $\alpha\in \lang$, either $\Mod(\Sigma)\subseteq \Mod(\{\alpha\})$ or $\Mod(\Sigma)\subseteq \mathbf{M}\setminus \Mod(\{\alpha\})$. Hence, $\Sigma$ is a complete set containing $\Gamma$.

\underline{(4) $\implies$ (1)}: Let $\Gamma\subseteq\lang$ be finitely satisfiable. Then, by (4), there exists a complete set $\Sigma$ such that $\Gamma\subseteq\Sigma$. Now, by \hyperlink{rem:comp=>sat}{\remref{rem:comp=>sat}}, $\Sigma$ is satisfiable, i.e., $\Mod(\Sigma)\neq\emptyset$. Since $\Gamma\subseteq\Sigma$, by \hyperlink{thm:cumulative(2)}{\thmref{thm:cumulative}(2)}, $\Mod(\Sigma)\subseteq\Mod(\Gamma)$. So, $\Mod(\Gamma)\neq\emptyset$, i.e., $\Gamma$ is satisfiable. Thus, every finitely satisfiable set is satisfiable. Hence, $\mstr$ is compact.

\underline{(1) $\implies$ (5)}: Let $\Gamma\subseteq\lang$ be $\deduc_{\mstr}$-trivial. Since $\lang$ is not finitely satisfiable, by (1), it is not satisfiable. Then, as $\mstr$ is normal, by \hyperlink{thm:nsat=>triv}{\thmref{thm:nsat=>triv}}, $\Gamma$ is not satisfiable. So, by (1) again, $\Gamma$ is not finitely satisfiable. Thus, there exists a finite $\Gamma_0\subseteq \Gamma$ that is not satisfiable. Again, by \hyperlink{thm:nsat=>triv}{\thmref{thm:nsat=>triv}}, this implies that $\Gamma_0$ is $\deduc_{\mstr}$-trivial. Thus, every $\deduc_\mstr$-trivial set contains a finite $\deduc_\mstr$-trivial set.

\underline{(5) $\implies$ (1)}: Suppose $\mstr$ is not compact. Then, there exists $\Gamma\subseteq \lang$ that is not satisfiable, but every finite subset of $\Gamma$ is satisfiable. Since $\Gamma$ is not satisfiable, by \hyperlink{thm:nsat=>triv}{\thmref{thm:nsat=>triv}}, $\Gamma$ is $\deduc_{\mstr}$-trivial. So, by (5), there exists a finite $\Gamma_0\subseteq \lang$ such that $\Gamma_0$ is $\deduc_{\mstr}$-trivial. We claim that $\Gamma_0$ is not satisfiable. 

Suppose the contrary, i.e., $\Gamma_0$ is satisfiable. So, by normality of $\mstr$, $\Gamma_0$ is finitely satisfiable; hence, by \hyperlink{thm:finsat=>maxfinsat}{\thmref{thm:finsat=>maxfinsat}}, is contained in a maximal finitely satisfiable set, say $\Sigma_0$. Then, $\Sigma_0\subsetneq\lang$ as $\lang$ is not finitely satisfiable. Now, by \hyperlink{thm:rep_Tarski}{\thmref{thm:rep_Tarski}}, $(\lang,\deduc_\mstr)$ is of Tarski-type. So, by monotonicity of $\deduc_\mstr$, as $\Gamma_0\subseteq\Sigma_0$ and $\Gamma_0$ is $\deduc_\mstr$-trivial, $\Sigma_0$ is also $\deduc_\mstr$-trivial. Thus, by \hyperlink{cor:nontriv=>sat}{\corref{cor:nontriv=>sat}}, $\Sigma_0$ is not satisfiable. Since $\mstr$ is normal, this implies that $\Sigma_0$ is not finitely satisfiable, which is a contradiction. Thus, $\Gamma_0$ cannot be satisfiable. 

So, we can conclude that there cannot be any set $\Gamma\subseteq\lang$ that is not satisfiable but finitely satisfiable. Hence, $\mstr$ is compact. 

\underline{(1) $\implies$ (6)}: Let $f:I\to \mathcal{P}(\lang)$ such that $I\ne\emptyset$ and the poset $(f(I),\subseteq)$ is $\subseteq$-directed. Suppose $f(i)$ is satisfiable for all $i\in I$ but $\displaystyle\bigcup_{i\in I} f(i)$ is not. Then, by (1), there exists a finite $\Lambda\subseteq\displaystyle\bigcup_{i\in I} f(i)$ that is not satisfiable. Now, by \hyperlink{thm:cumulative(1)}{\thmref{thm:cumulative}(1)}, $\Mod(\emptyset)=\mathbf{M}$ and so, $\emptyset$ is satisfiable. Thus, $\Lambda\ne\emptyset$. Let $\Lambda=\{\lambda_1,\ldots,\lambda_n\}$ and $\lambda_t\in f(i_t)$ for each $1\le t\le n$. So, $\Lambda\subseteq  \displaystyle\bigcup_{t=1}^n f(i_t)$. Now, as $(f(I),\subseteq)$ is $\subseteq$-directed, $\displaystyle\bigcup_{t=1}^n f(i_t)\subseteq f(k)$ for some $k\in I$. This implies that $\Lambda\subseteq f(k)$. Now, as $\mstr$ is normal and $\Lambda$ is not satisfiable, $f(k)$ is not satisfiable as well. However, by our hypothesis, $f(k)$ must be satisfiable as $k\in I$. This is a contradiction. Hence, $\displaystyle\bigcup_{i\in I}f(i)$ is satisfiable.

\underline{(6) $\implies$ (7)}: Let $\Sigma\subseteq\lang$ be finitely satisfiable and $f:\Finset(\Sigma)\to\pow(\lang)$ an order-preserving map between the posets $(\Finset(\Sigma),\subseteq)$ and $(\pow(\lang),\subseteq)$. Clearly, $(\Finset(\Sigma),\subseteq)$ is $\subseteq$-directed. Now, as $f$ is order-preserving, the poset $(f(\Finset(\Sigma)),\subseteq)$ is also $\subseteq$-directed. Then, by (6), if $f(\Gamma)$ is satisfiable for each $\Gamma\in\Finset(\Sigma)$, then so is $\displaystyle\bigcup_{\Gamma\in \Finset(\Sigma)} f(\Gamma)$.

\underline{(7) $\implies$ (1)}: Let $\Sigma\subseteq \lang$ be finitely satisfiable and $\iota_\Sigma:\Finset(\Sigma)\to \mathcal{P}(\lang)$ be the inclusion map, i.e., $\iota_\Sigma(\Gamma)=\Gamma$ for all $\Gamma\in \Finset(\Sigma)$. Clearly, $\iota_\Sigma$ is an order-preserving map between the posets $(\Finset(\Sigma),\subseteq)$ and $(\pow(\lang),\subseteq)$. Since $\Sigma$ is finitely satisfiable, $\iota_\Sigma(\Gamma)=\Gamma$ is satisfiable for all $\Gamma\in \Finset(\Sigma)$. So, by (7), $\displaystyle\bigcup_{\Gamma\in \Finset(\Sigma)}\iota_{\Sigma}(\Gamma)= \displaystyle\bigcup_{\Gamma\in \Finset(\Sigma)}\Gamma=\Sigma$ is also satisfiable. Thus, every finitely satisfiable set is satisfiable. Hence, $\mstr$ is compact.

\underline{(6) $\implies$ (8)}: Let $f:I\to \mathcal{P}(\mathbf{M})$ such that $I\ne\emptyset$ and the poset $(f(I),\supseteq)$ is $\supseteq$-directed. Suppose $f(i)\ne\emptyset$ for all $i\in I$. We recall that $\Th:\pow(\mathbf{M})\to\pow(\lang)$; so, $\Th\circ f:I\to \mathcal{P}(\lang)$. Now, let $\Th(f(i)),\Th(f(j))\in\Th(f(I))$, where $i,j\in I$. Then, $f(i),f(j)\in f(I)$. Since $(f(I),\supseteq)$ is $\supseteq$-directed, there exists $f(k)\in f(I)$, for some $k\in I$, such that $f(k)\supseteq f(i)$ and $f(k)\supseteq f(j)$. Then, \hyperlink{thm:cumulative_th(2)}{\thmref{thm:cumulative_th}(2)}, $\Th(f(i))\subseteq\Th(f(k))$ and $\Th(f(j))\subseteq\Th(f(k))$. This implies that the poset $((\Th\circ f)(I),\subseteq)$ is $\subseteq$-directed.

We now show that $(\Th\circ f)(i)$ is satisfiable for all $i\in I$. By our hypothesis, $f(i)\neq\emptyset$ for all $i\in I$. Let $m_i\in f(i)$ for each $i\in I$. So, by \hyperlink{thm:cumulative_th(2)}{\thmref{thm:cumulative_th}(2)}, $\Th(f(i))\subseteq \Th(\{m_i\})$. Then, as $m_i\models \Th(\{m_i\})$ for all $i\in I$, by normality of $\mstr$, $m_i\models\Th(f(i))$ for all $i\in I$. Thus, $\Th(f(i))$ is satisfiable for all $i\in I$. Hence, by (6), $\displaystyle\bigcup_{i\in I}\Th(f(i))$ is satisfiable. 

\underline{(8) $\implies$ (9)}: Let $\Sigma\subseteq\lang$ be finitely satisfiable and $f:\Finset(\Sigma)\to\pow(\mathbf{M})$ an order-reversing map between the posets $(\Finset(\Sigma),\subseteq)$ and $(\pow(\mathbf{M}),\subseteq)$. Clearly, $(\Finset(\Sigma),\subseteq)$ is $\subseteq$-directed. Now, as $f$ is order-reversing, the poset $(f(\Finset(\Sigma)),\supseteq)$ is $\supseteq$-directed. Then, by (8), if $f(\Gamma)\neq\emptyset$ for each $\Gamma\in\Finset(\Sigma)$, then $\displaystyle\bigcup_{\Gamma\in \Finset(\Sigma)} \Th(f(\Gamma))$ is satisfiable.

\underline{(9) $\implies$ (1)}: Let $\Sigma\subseteq\lang$ be finitely satisfiable and $\mu_\Sigma:\Finset(\Sigma)\to \mathcal{P}(\mathbf{M})$ be defined as follows. For each $\Sigma_0\in\Finset(\Sigma)$, $\mu_{\Sigma}(\Sigma_0)=\Mod(\Sigma_0)$. Then, by \hyperlink{thm:cumulative(2)}{\thmref{thm:cumulative}(2)}, $\mu_\Sigma$ is an order-reversing map between the posets $(\Finset(\Sigma),\subseteq)$ and $(\pow(\mathbf{M},\subseteq)$. Now, as $\Sigma$ is finitely satisfiable, $\mu_{\Sigma}(\Sigma_0)\ne\emptyset$ for all $\Sigma_0\in \Finset(\Sigma)$. So, by (9), $\displaystyle\bigcup_{\Sigma_0\in \Finset(\Sigma)}\Th(\mu_\Sigma(\Sigma_0))$ is satisfiable. Since $\mstr$ is normal, by \hyperlink{thm:cumulative_th(5)}{\thmref{thm:cumulative_th}(5)}, $\Sigma_0\subseteq\Th(\Mod(\Sigma_0))$ for all $\Sigma_0\in \Finset(\Sigma)$. Thus, $\Sigma=\displaystyle \bigcup_{\Sigma_0\in \Finset(\Sigma)}\Sigma_0\subseteq \displaystyle \bigcup_{\Sigma_0\in \Finset(\Sigma)}\Th(\mu_\Sigma(\Sigma_0))$. Again, by normality of $\mstr$, this implies that $\Sigma$ is satisfiable. Hence, $\mstr$ is compact.
\end{proof}

The normality condition clearly plays a crucial role in establishing the equivalence between the statements in the above theorem. The following example further illustrates this point by providing a scenario where this equivalence breaks in the absence of normality.

\begin{ex}
We consider the $\amst$ $\mstr=(\mathbb{N},\models,\mathcal{P}(\mathbb{N}))$, where $\mathbb{N}$ is the set of natural numbers and $\models\,\subseteq\mathbb{N}\times\pow(\mathbb{N})$ is defined as follows. For all $\Gamma\cup\{n\}\subseteq \mathbb{N}$, 
\[
n\models \Gamma \quad\hbox{ iff }\quad \Gamma\ne\{0\},\,\{2t+1:t\in \mathbb{N}\},\,\mathbb{N}.
\]
Since $\{0\}$ is not satisfiable, $\mathbb{N}$ is not finitely satisfiable. $\mstr$ is, however, not normal since, for any $n\in\mathbb{N}$, $n\models\{0,1\}$ but $n\not\models\{0\}$.

Now, let $\Gamma\subseteq \mathbb{N}$ be finitely satisfiable. Since $\mathbb{N}$ is not finitely satisfiable, $\Gamma\subsetneq \mathbb{N}$. Let $k\in \mathbb{N}\setminus \Gamma$. Then, $\mathbb{N}\setminus \{k\}$ is satisfiable and since $\mathbb{N}$ is not satisfiable, is, in fact, maximal satisfiable. Thus, every finitely satisfiable set is contained in a maximal satisfiable set, i.e., statement (3) of \hyperlink{thm:compact_normal_ams(I)}{\thmref{thm:compact_normal_ams(I)}} holds for $\mstr$. However, $\mstr$ is not compact since while $\{2t+1:t\in \mathbb{N}\}$ is finitely satisfiable, it is not satisfiable. Thus, statement (1) of \hyperlink{thm:compact_normal_ams(I)}{\thmref{thm:compact_normal_ams(I)}} fails for $\mstr$.
\end{ex}

Suppose $(\lang,\deduc)$ is a Tarski-type logical structure. Then, by \hyperlink{thm:rep_Tarski}{\thmref{thm:rep_Tarski}}, there exists a normal $\amst$ $\mstr$ such that $\lang$ is not satisfiable in $\mstr$ and $\deduc\,=\,\deduc_\mstr$. We now raise the following question. If, moreover, $\mstr$ is compact, and $\deduc\,=\,\deduc_{\mathfrak{N}}$ for some normal $\amst$ $\mathfrak{N}$, then is $\mathfrak{N}$ also compact? The following result shows that this is indeed the case.

\begin{cor}
    Suppose $\mstr_1=(\mathbf{M}_1,\models_1,\mathcal{P}(\lang)),\,\mstr_2=(\mathbf{M}_2,\models_2,\mathcal{P}(\lang))$ are two normal $\amst$s such that $\lang$ is not satisfiable in either of them, and $(\lang,\deduc_{\mstr_1}),\,(\lang,\deduc_{\mstr_2})$ the logical structures induced by them, respectively. If $\mstr_1$ is compact and $\deduc_{\mstr_1}\,=\,\deduc_{\mstr_2}$, then $\mstr_2$ is compact as well. 
\end{cor}

\begin{proof}
    Suppose $\mstr_2$ is not compact. Now, as $\mstr_2$ is normal, every satisfiable set is finitely satisfiable in $\mstr_2$. Thus, there exists $\Gamma\subseteq \lang$ such that $\Gamma$ is finitely satisfiable but not satisfiable in $\mstr_2$. Since $\Gamma$ is not satisfiable in $\mstr_2$, by \hyperlink{thm:nsat=>triv}{\thmref{thm:nsat=>triv}}, $\Gamma$ is $\deduc_{\mstr_2}$-trivial. Then, as $\deduc_{\mstr_1}\,=\,\deduc_{\mstr_2}$, $\Gamma$ is $\deduc_{\mstr_1}$-trivial as well. Since $\mstr_1$ is compact and $\lang$ is not satisfiable in $\mstr_1$, $\lang$ is not finitely satisfiable in $\mstr_1$ as well. Thus, $\mstr_1$ is a normal $\amst$ such that $\lang$ is not finitely satisfiable in it. Then, since $\mstr_1$ is compact, by \hyperlink{thm:compact_normal_ams(I)(5)}{\thmref{thm:compact_normal_ams(I)}(5)}, there exists a finite $\Gamma_0\subseteq \Gamma$ that is $\deduc_{\mstr_1}$-trivial. As $\deduc_{\mstr_1}\,=\,\deduc_{\mstr_2}$, $\Gamma_0$ is $\deduc_{\mstr_2}$-trivial as well. Now, as $\mstr_2$ is normal and $\lang$ is not satisfiable in $\mstr_2$, by \hyperlink{thm:nsat=>triv}{\thmref{thm:nsat=>triv}} again, $\Gamma_0$ is not satisfiable in $\mstr_2$. This, however, contradicts the finite satisfiability of $\Gamma$ in $\mstr_2$. Hence, $\mstr_2$ is compact.
\end{proof}

\begin{defn}
    A logical structure $(\lang,\deduc)$ is said to be \emph{finitary} if, for all $\Gamma\cup\{\alpha\}\subseteq\lang$, $\Gamma\deduc\alpha$ implies that there exists a finite $\Gamma_0\subseteq\Gamma$ such that $\Gamma_0\deduc\alpha$.  
\end{defn}

\hypertarget{thm:finitary}{\begin{thm}{\label{thm:finitary}}
    Suppose $(\lang,\deduc)$ is a finitary logical structure such that $\deduc$ satisfies monotonicity and transitivity. If there exists a nonempty finite $\deduc$-trivial set, then every $\deduc$-trivial set contains a finite $\deduc$-trivial set.
\end{thm}}

\begin{proof}
    Suppose $\Lambda=\{\lambda_1,\ldots,\lambda_n\}\subseteq\lang$ is a nonempty $\deduc$-trivial set. Let $\Gamma\subseteq\lang$ be a $\deduc$-trivial set. So, $\Gamma\deduc\lambda_i$ for all $1\le i\le n$. Now, since $(\lang,\deduc)$ is finitary, for each $1\le i\le n$, there exists a finite $\Gamma_i\subseteq\Gamma$ such that $\Gamma_i\deduc\lambda_i$. Let $\Sigma=\displaystyle\bigcup_{i=1}^n\Gamma_i$. Then, by monotonicity of $\deduc$, $\Sigma\deduc\lambda_i$ for all $1\le i\le n$. So, by transitivity of $\deduc$, since $\Lambda$ is $\deduc$-trivial, $\Sigma$ is $\deduc$-trivial as well. Thus, $\Sigma$ is a finite $\deduc$-trivial subset of $\Gamma$.
\end{proof}

Using the above theorem and \hyperlink{thm:compact_normal_ams(I)}{\thmref{thm:compact_normal_ams(I)}}, we now obtain the following sufficient condition for a normal $\amst$ to be compact. An interested reader might compare this with \cite[Theorem 2.17]{LewitzkaBrunner2009}.

\begin{cor}
    Suppose $\mstr=(\mathbf{M},\models,\mathcal{P}(\lang))$ is a normal $\mathsf{amst}$ and $(\lang,\deduc_{\mstr})$ the logical structure induced by $\mstr$ such that $(\lang,\deduc_\mstr)$ is finitary. If $\lang$ is not finitely satisfiable, then $\mstr$ is compact.
\end{cor}
\begin{proof}
    Since $\mstr$ is normal, by \hyperlink{thm:rep_Tarski}{\thmref{thm:rep_Tarski}}, $(\lang,\deduc_\mstr)$ is of Tarski-type and hence, $\deduc_\mstr$ satisfies monotonicity and transitivity. Now, as $\lang$ is not finitely satisfiable, there exists a finite $\Lambda\subseteq \lang$ that is not satisfiable. Then, by \hyperlink{thm:nsat=>triv}{\thmref{thm:nsat=>triv}}, $\Lambda$ is $\deduc_\mstr$-trivial. So, by \hyperlink{thm:finitary}{\thmref{thm:finitary}}, every $\deduc_\mstr$-trivial set contains a finite $\deduc_\mstr$-trivial set. Hence, by \hyperlink{thm:compact_normal_ams(I)}{\thmref{thm:compact_normal_ams(I)}}, $\mstr$ is compact.
\end{proof}

{\hypertarget{subsec:com-top}{\subsection{Compactness and  Topology}}\label{subsec:com-top}}

In this section, we prove another characterization of compact normal $\mathsf{amst}$s via topological methods. The idea of such topological proofs dates back, at least, to 1952 in \cite{Tarski1952} and continues to the present day (see, e.g., \cite{Caicedo1993, Caicedo1999,Lewitzka2005, Lewitzka2007, LewitzkaBrunner2009}). The main steps of such an argument are presented below for classical propositional logic. Before going into that, we first recall some basic definitions and a theorem (without proof) of point-set topology. These can be found in any standard textbook on topology, such as \cite{Kelley1975, Munkres2000}.   

A \emph{topology} on a set $X$ is a collection $\tau$ of subsets of $X$, called the \emph{$\tau$-open subsets} of $X$, such that the following properties are satisfied.
        \begin{enumerate}[label=(\alph*)]
            \item $\emptyset$ and $X$ are $\tau$-open.
            \item If $(U_i)_{i\in I}$ is a family of $\tau$-open subsets of $X$, then so is $\displaystyle\bigcup_{i\in I}U_i$. 
            \item If $(U_i)_{i\in I}$ is a finite family of $\tau$-open subsets of $X$, then so is $\displaystyle\bigcap_{i\in I}U_i$. 
        \end{enumerate}
        The pair $(X,\tau)$ is then called a \emph{topological space}.
        
        A set $C\subseteq X$ is said to be \emph{$\tau$-closed} if $C=X\setminus U$ for some $U\in \tau$.

Given a set $X$, a \emph{base} on $X$ is a collection $\beta$ of subsets of $X$ such that the following properties are satisfied.
        \begin{enumerate}[label=(\alph*)]
            \item $X\subseteq \displaystyle\bigcup_{B\in \beta}B$.
            \item If $U,V\in\beta$, then there exists $W\in \beta$ such that $W\subseteq U\cap V$.
        \end{enumerate}

Given a set $X$, a \emph{subbase/cover} on $X$ is a collection $\sigma$ of subsets of $X$, such that $X\subseteq \displaystyle\bigcup_{S\in \sigma}S$.

Moreover, if $\sigma$ is finite, then it is called a \textit{finite cover} on $X$. Furthermore, any $\rho\subseteq \sigma$ that is a cover of $X$ is called a \textit{subcover} of $\sigma$.

Suppose $X$ is a set and $\beta$ is a base on $X$. Then, the following collection of subsets of $X$ forms a topology on $X$. 
        \[
        \tau_\beta=\left\{U\subseteq X\mid\,\displaystyle\bigcup_{B\in \gamma}B=U~\text{for some}~\gamma\subseteq \beta\right\}
        \]
The topology $\tau_\beta$ is called the \emph{topology generated by $\beta$}.

Suppose $X$ is a set and $\sigma$ is a subbase on $X$. Let 
        \[
        \begin{array}{c}
        \sigma_\beta=\left\{U\subseteq X\mid\,\displaystyle\bigcap_{S\in \xi}S= U\hbox{ for some finite }\xi\subseteq\sigma\right\}\\
        \hbox{and}\\
        \tau_\sigma=\left\{U\subseteq X\mid\,\displaystyle\bigcup_{B\in \gamma}B=U\hbox{ for some }\gamma\subseteq \sigma_\beta\right\}
        \end{array}
        \]
Then, $\tau_\sigma$ forms a topology on $X$ and is called the \emph{topology generated by $\sigma$}. In this case, the collection $\sigma_\beta$ is a base for a topology on $X$ and, in fact, $\tau_\sigma$ is the topology generated by $\sigma_\beta$.

Given a topological space $(X,\tau)$ and a base $\beta$ on $X$, $\beta$ is called a \emph{base for $(X,\tau)$}, if the topology generated by $\beta$ coincides with $\tau$, i.e., $\tau=\tau_\beta$. Similarly, a subbase $\sigma$ on $X$ is said to be a \emph{subbase for $(X,\tau)$} if the topology generated by $\sigma$ coincides with $\tau$, i.e., $\tau=\tau_\sigma$.     

Let $(X,\tau)$ be a topological space. A cover $\lambda$ on $X$ is said to be a \emph{$\tau$-open cover} (respectively, \emph{subbasic $\tau$-open cover}) of $X$, if $\lambda\subseteq \tau$ (respectively, $\lambda\subseteq \sigma$, where $\sigma$ is a subbase for $(X,\tau)$).

\hypertarget{def:top_comp}{\begin{defn}[\textsc{Compact Space}]{\label{def:top_comp}}
  A topological space $(X,\tau)$ is said to be \emph{compact} if every $\tau$-open cover $\sigma$ on $X$ has a finite subcover.
\end{defn}}

\hypertarget{thm:char_compsp}{\begin{thm}[\textsc{Alexander Subbase Theorem}]{\label{thm:char_compsp}}
Suppose $(X,\tau)$ is a topological space and $\sigma$ a subbase for it. Then, $(X,\tau)$ is compact iff every subbasic $\tau$-open cover has a finite subcover.
\end{thm}}

We now briefly go through the topological proof of the compactness theorem for classical propositional logic (\textbf{CPL}). We assume that $\lang_\mathbf{CPL}$ is the set of well-formed formulas (wffs) of \textbf{CPL} constructed in the usual way over a countable set of propositional variables $V$ using the unary connective $\neg$ and the binary connectives $\to,\land$ and $\lor$. A \emph{truth assignment} is a function $v:V\to\{0,1\}$. Each truth assignment $v$ can then be uniquely extended to a \emph{valuation} $\overline{v}:\lang\to\{0,1\}$.

Let $\mathcal{V}$ denote the set of all truth assignments. The semantic consequence $\models_{\mathbf{CPL}}\subseteq\mathcal{V}\times \mathscr{L}_{\textbf{CPL}}$ is then defined as follows. For any $v\in\mathcal{V}$ and $\varphi\in\lang_\textbf{CPL}$, $v\models_\textbf{CPL}\varphi$ iff $\overline{v}(\varphi)=1$. Now, for each $\varphi\in \mathscr{L}_{\textbf{CPL}}$, we define $U(\varphi)=\{v\in\mathcal{V}\mid\,\overline{v}(\varphi)=1\}$. Then, for all $\varphi_1,\varphi_2\in\mathscr{L}_{\mathbf{CPL}}$,
\begin{align*}
U(\varphi_1)\cap U(\varphi_2)&=\{v\in\mathcal{V}\mid\,\overline{v}(\varphi_1)=1\}\cap \{v\in\mathcal{V}\mid\,\overline{v}(\varphi_2)=1\}\\&=\{v\in\mathcal{V}\mid\,\overline{v}(\varphi_1\land \varphi_2)=1\}\\&=U(\varphi_1\land \varphi_2).
\end{align*}
Hence, the set $\{U(\varphi)\mid\,\varphi\in \mathscr{L}_{\mathbf{CPL}}\}$ forms a base on $\mathcal{V}$. Let $\tau_{\mathbf{CPL}}$ be the topology generated by it.

We note that, if $\Sigma\subseteq \mathscr{L}_{\mathbf{CPL}}$ is satisfiable, then it is finitely satisfiable. So, to prove the compactness theorem for \textbf{CPL}, it is enough to show that if $\Sigma\subseteq \mathscr{L}_{\mathbf{CPL}}$ is finitely satisfiable, then it is satisfiable as well. If possible, suppose $\Sigma$ is finitely satisfiable, but not satisfiable. Now, 
\begin{align*}
    \Sigma~\text{is not satisfiable}&\iff~\text{for all}~v\in\mathcal{V},~\text{there exists}~\varphi\in \Sigma~\text{such that}~\overline{v}(\varphi)=0\\&\iff~\text{for all}~v\in \mathcal{V},~\text{there exists}~\varphi\in \Sigma~\text{such that}~\overline{v}(\neg\varphi)=1\\&\iff~\text{for all}~v\in \mathcal{V},~\text{there exists}~\varphi\in \Sigma~\text{such that}~v\in U(\neg\varphi)\\&\iff \mathcal{V}\subseteq \displaystyle\bigcup_{\varphi\in \Sigma}U(\neg\varphi).
\end{align*}
Therefore, it follows that, the $\amst$ $(\mathcal{V},\models_{\mathbf{CPL}},\mathscr{L}_{\mathbf{CPL}})$ is compact (in the sense of \hyperlink{def:comp}{\defref{def:comp}}) iff $(\mathcal{V},\tau_{\mathbf{CPL}})$ is compact (in the sense of \hyperlink{def:top_comp}{\defref{def:top_comp}}). The argument is then completed by showing that $(\mathcal{V},\tau_{\mathbf{CPL}})$ is indeed compact.  

Thus, essentially, given an $\amst$, the arguments in such proofs can be broken down into the following steps:
\begin{enumerate}[label=(\alph*)]
    \item construction of a topological space from the $\amst$ under consideration,
    \item proving that the $\amst$ under consideration is compact  (in the sense of \hyperlink{def:comp}{\defref{def:comp}}) iff the associated topological space is compact (in the sense of \hyperlink{def:top_comp}{\defref{def:top_comp}}), and finally,
    \item proving that the topological space is indeed compact.
\end{enumerate}
We follow a similar line of argument. However, before entering into the details, we point out some key differences between our approach and those of the papers cited above.
\begin{enumerate}[label=$\bullet$]
    \item In \cite{Caicedo1993, Caicedo1999}, although some characterizations of general versions of (topological) compactness are discussed, the application of the results is restricted only to `model-theoretic logics' in the sense of \cite{Ebbinghaus1985, Lindstrom1969} and the proofs are often syntax-dependent (see, e.g., \cite[Lemma 2.3]{Caicedo1999}). Moreover, the canonical topological spaces associated with such logics are unlike the ones we investigate (cf. \cite[Section II]{Caicedo1999} and the topology defined in \hyperlink{thm:compact_normal_ams(II)}{\thmref{thm:compact_normal_ams(II)}}). 
    \item In \cite{Lewitzka2005, Lewitzka2007, LewitzkaBrunner2009}, topological approaches to the study of `syntax-free abstract logics' have been carried out in great detail. However, relatively little attention has been given to the form of compactness with which we are interested in the present exposition. It may be observed that \cite[Theorem 2.14]{LewitzkaBrunner2009} follows from \hyperlink{thm:compact_normal_ams(I)(6)}{\thmref{thm:compact_normal_ams(I)}(6)}. Another theorem which can, perhaps, be seen as connected with the contents of this paper is \cite[Corollary 2.18]{LewitzkaBrunner2009} (\hyperlink{thm:compact_normal_ams(I)(6)}{\thmref{thm:compact_normal_ams(I)}(7)} is a form of this). Moreover, the topological space associated with an `abstract logic,' henceforth referred to as an \emph{LM-type abstract logic} (see \cite[Definition 2.1]{LewitzkaBrunner2009} for details), is also different from the ones we investigate (cf., e.g., \cite[Definition 3.5]{LewitzkaBrunner2009} and the topology defined in \hyperlink{thm:compact_normal_ams(II)}{\thmref{thm:compact_normal_ams(II)}}). However, each normal $\amst$ induces an $LM$-type abstract logic and vice versa. The results of this section thus complement those of the aforementioned papers. 
\end{enumerate}

Our second characterization of compact normal $\amst$s is as follows.

\hypertarget{thm:compact_normal_ams(II)}{\begin{thm}[\textsc{Compactness for Normal $\mathsf{amst}$ (II)}]{\label{thm:compact_normal_ams(II)}}
Suppose $\mstr=(\mathbf{M},\models,\mathcal{P}(\lang))$ is a normal $\amst$ such that $\lang$ is not satisfiable. Let $\tau_N$ be the topology generated by $\{\mathbf{M}\setminus\Mod(\{\alpha\})\mid\,\alpha\in\lang\}$ as a subbase. Then, $\mstr$ is compact iff $(\mathbf{M},\tau_N)$ is a compact topological space.
\end{thm}}

\begin{proof}
We first show that the collection $\{\mathbf{M}\setminus\Mod(\{\alpha\})\mid\,\alpha\in\lang\}$ is a subbase on $\mathbf{M}$. Suppose the contrary. Then, there exists $m\in\mathbf{M}$ such that $m\notin\displaystyle\bigcup\{\mathbf{M}\setminus\Mod(\{\alpha\})\mid\,\alpha\in\lang\}$. This implies that $m\in\Mod(\{\alpha\})$ for all $\alpha\in\lang$, i.e., $m\models\{\alpha\}$ for all $\alpha\in\lang$. Since $\mstr$ is normal, this shows that $m\models\lang$. So, $\lang$ is satisfiable. This, however, is a contradiction. Thus, $\mathbf{M}\subseteq\displaystyle\bigcup\{\mathbf{M}\setminus\Mod(\{\alpha\})\mid\,\alpha\in\lang\}$ and so, $\{\mathbf{M}\setminus\Mod(\{\alpha\})\mid\,\alpha\in\lang\}$ is a subbase on $\mathbf{M}$.

Suppose $\mstr$ is compact but $(\mathbf{M},\tau_N)$ is not a compact topological space. Let $\lambda=\{\mathbf{M}\setminus\Mod(\{\beta_i\})\mid\,i\in I\}$ be a subbasic $\tau_N$-open cover of $\mathbf{M}$ that has no finite subcover and $\Sigma=\{\beta_i:i\in I\}$. 

We first show that $\Sigma$ is finitely satisfiable as follows. Let $\{\beta_0,\ldots,\beta_n\}\subseteq \Sigma$. Since $\lambda$ has no finite subcover, $\displaystyle\bigcup_{i=0}^n\mathbf{M}\setminus \Mod(\{\beta_i\})\ne\mathbf{M}$, or equivalently,  $\displaystyle\bigcap_{i=0}^n\Mod(\{\beta_i\})\ne\emptyset$. Now, by \hyperlink{thm:cumulative(1)}{\thmref{thm:cumulative}(1)}, $\Mod(\{\beta_0,\ldots,\beta_n\})=\displaystyle\bigcap_{i=0}^n\Mod(\{\beta_i\})$. So, $\Mod(\{\beta_0,\ldots,\beta_n\})\neq\emptyset$, i.e., $\{\beta_0,\ldots,\beta_n\}$ is satisfiable. Hence, $\Sigma$ is finitely satisfiable. Then, as $\mstr$ is compact, $\Sigma$ is satisfiable as well. So, $\Mod(\Sigma)\neq\emptyset$. Now, by \hyperlink{thm:cumulative(1)}{\thmref{thm:cumulative}(1)}), $\Mod(\Sigma)=\displaystyle\bigcap_{i\in I}\Mod(\{\beta_i\})$. Thus, $\displaystyle\bigcap_{i\in I}\Mod(\{\beta_i\})\ne\emptyset$, which implies that $\mathbf{M}\setminus \displaystyle\bigcap_{i\in I}\Mod(\{\beta_i\})=\displaystyle\bigcup_{i\in I}(\mathbf{M}\setminus \Mod(\{\beta_i\}))=\bigcup\lambda\ne \mathbf{M}$. This contradicts the assumption that $\lambda$ covers $\mathbf{M}$. Thus, every subbasic $\tau_N$-open cover of $\mathbf{M}$ has a finite subcover. Hence, by \hyperlink{thm:char_compsp}{\thmref{thm:char_compsp}}, $(\mathbf{M},\tau_N)$ is a compact topological space.  

Conversely, suppose $(\mathbf{M},\tau_N)$ is a compact topological space but $\mstr$ is not compact. Then, there exists $\Sigma\subseteq \lang$ such that $\Mod(\Sigma)=\emptyset$, while $\Mod(\Sigma_0)\ne\emptyset$ for all finite $\Sigma_0\subseteq\Sigma$. 
Now, by \hyperlink{thm:cumulative(1)}{\thmref{thm:cumulative}(1)}, $\Mod(\Sigma)=\displaystyle\bigcap_{\beta\in \Sigma}\Mod(\{\beta\})$. So, $\displaystyle\bigcap_{\beta\in \Sigma}\Mod(\{\beta\})=\emptyset$ and hence, $\displaystyle\bigcup_{\beta\in\Sigma}(\mathbf{M}\setminus \Mod(\{\beta\}))=\mathbf{M}\setminus \displaystyle\bigcap_{\beta\in\Sigma}\Mod(\{\beta\})=\mathbf{M}$. So, $\{\mathbf{M}\setminus\Mod(\{\beta\})\mid\,\beta\in \Sigma\}$ covers $\mathbf{M}$. Moreover, $\{\mathbf{M}\setminus\Mod(\{\beta\})\mid\,\beta\in \Sigma\}$ is a subbasic $\tau_N$-open cover of $\mathbf{M}$. Since $(\mathbf{M},\tau_N)$ is a compact topological space, by \hyperlink{thm:char_compsp}{\thmref{thm:char_compsp}}, there exists a finite subcover of it. Let $\{\mathbf{M}\setminus\Mod(\{\beta_0\}),\ldots, \mathbf{M}\setminus\Mod(\{\beta_n\})\}$ be such a finite subcover. Then, $\displaystyle\bigcup_{i=0}^n(\mathbf{M}\setminus\Mod(\{\beta_i\}))=\mathbf{M}$ and hence, $\displaystyle\bigcap_{i=0}^n\Mod(\{\beta_i\})=\emptyset$. Again, by \hyperlink{thm:cumulative(1)}{\thmref{thm:cumulative}(1)}, $\Mod(\{\beta_0,\ldots,\beta_n\})=\displaystyle\bigcap_{i=0}^n\Mod(\{\beta_i\})$. Thus, $\Mod(\{\beta_0,\ldots,\beta_n\}=\emptyset$. This is a contradiction since $\{\beta_0,\ldots,\beta_n\}$ is a finite subset of $\Sigma$. Hence, $\mstr$ is compact.
\end{proof}

Given an $\amst$ $\mstr=(\mathbf{M},\models,\pow(\lang))$, there are two maps $\Mod:\mathcal{P}(\lang)\to \mathcal{P}(\mathbf{M})$ and $\Th:\mathcal{P}(\mathbf{M})\to \mathcal{P}(\lang)$. We now consider the composite of these, viz., $\Mod\circ \Th:\mathcal{P}(\mathbf{M})\to \mathcal{P}(\mathbf{M})$ and end this section with an alternative description of the topology described in \hyperlink{thm:compact_normal_ams(II)}{\thmref{thm:compact_normal_ams(II)}} in terms of this map. 

\hypertarget{thm:Tarski_top}{\begin{thm}{\label{thm:Tarski_top}}
    Suppose $\mstr=(\mathbf{M},\models,\mathcal{P}(\lang))$ is a normal $\mathsf{amst}$ such that $\lang$ is not satisfiable. Let $\tau_N$ be the topology described in \hyperlink{thm:compact_normal_ams(II)}{\thmref{thm:compact_normal_ams(II)}}. Then, any set of the form $Z=\Mod(\Sigma)$, where $\Sigma\subseteq\lang$, is a $\tau_N$-closed set. Moreover, for any such $\tau_N$-closed $Z$, $(\Mod\circ\Th)(Z)=Z$.
\end{thm}}

\begin{proof}
    We first show that for all $M\cup N\subseteq \mathbf{M}$ the following statements hold. 
    \begin{enumerate}[label=(\alph*)]
        \item $M\subseteq (\Mod\circ\Th)(M)$.
        \item If $M\subseteq N$, then $(\Mod\circ\Th)(M)\subseteq (\Mod\circ\Th)(N)$.
        \item $(\Mod\circ\Th)((\Mod\circ\Th)(M))=(\Mod\circ\Th)(M)$.        
    \end{enumerate}
    The above statements, in fact, amount to showing that $(\mathbf{M},\Mod\circ\Th)$ is a Tarski-type logical structure in the sense described in \hyperlink{rem:operator}{\remref{rem:operator}}. 
    \begin{enumerate}[label=(\alph*)]
        \item Let $m\in\mathbf{M}$ such that $m\notin\Mod(\Th(M))$. Then, $m\not\models\Th(M)$. Since $\mstr$ is normal, this implies that $m\not\models\{\alpha\}$ for some $\alpha\in\Th(M)$. So, $m\notin M$. Hence, $M\subseteq(\Mod\circ\Th)(M)$.
        \item Let $M\subseteq N$. So, by \hyperlink{thm:cumulative_th(2)}{\thmref{thm:cumulative_th}(2)}, $\Th(N)\subseteq \Th(M)$. Then, by \hyperlink{thm:cumulative(2)}{\thmref{thm:cumulative}(2)}, $\Mod(\Th(M))\subseteq \Mod(\Th(N))$. Hence, $(\Mod\circ\Th)(M)\subseteq (\Mod\circ\Th)(N)$.
        \item We note that by \hyperlink{thm:cumulative_th(5)}{\thmref{thm:cumulative_th}(5)}, $\Th(M)\subseteq\Th(\Mod(\Th(M)))$. So, by \hyperlink{thm:cumulative(2)}{\thmref{thm:cumulative}(2)}, $\Mod(\Th((\Mod(\Th(M)))))\subseteq\Mod(\Th(M))$. The reverse inclusion follows by statement (a). Thus, $(\Mod\circ\Th)((\Mod\circ\Th)(M))=(\Mod\circ\Th)(M)$.
    \end{enumerate}
    
    Now, let $Z=\Mod(\Sigma)$ for some $\Sigma\subseteq\lang$. Then, by \hyperlink{thm:cumulative(3)}{\thmref{thm:cumulative}(3)},
    \[
        Z=\Mod(\Sigma)=\Mod\left(\displaystyle\bigcup_{\alpha\in\Sigma}\{\alpha\}\right)=\bigcap_{\alpha\in\Sigma}\Mod(\{\alpha\}).
    \]
    Thus, $\mathbf{M}\setminus Z=\mathbf{M}\setminus\displaystyle\bigcap_{\alpha\in\Sigma}\Mod(\{\alpha\})=\displaystyle\bigcup_{\alpha\in\Sigma}\left(\mathbf{M}\setminus\Mod(\{\alpha\})\right)$, which is a union of elements from the subbasis generating $\tau_N$. Since elements of a subbasis are open and union of open sets is open, $\mathbf{M}\setminus Z$ is open, i.e., $Z$ is $\tau_N$-closed.
    
    We next note that $Z\subseteq(\Mod\circ\Th)(Z)$ by statement (b) above. Now, 
    \[
    (\Mod\circ\Th)(Z)=(\Mod\circ\Th)(\Mod(\Sigma))=\Mod(\Th(\Mod(\Sigma)))
    \]
    Since $\mstr$ is normal, by \hyperlink{thm:cumulative_th(5)}{\thmref{thm:cumulative_th}(5)}, $\Sigma\subseteq\Th(\Mod(\Sigma))$. Then, by \hyperlink{thm:cumulative(2)}{\thmref{thm:cumulative}(2)}, $\Mod(\Th(\Mod(\Sigma)))\subseteq\Mod(\Sigma)$, i.e., $(\Mod\circ\Th)(Z)\subseteq Z$. Hence, $(\Mod\circ\Th)(Z)= Z$.
\end{proof}

{\hypertarget{subsec:comp-ultra}{\subsection{Compactness and Ultramodels}\label{subsec:comp-ultra}}}

 We note that \hyperlink{thm:compact_normal_ams(I)}{\thmref{thm:compact_normal_ams(I)}} and \hyperlink{thm:compact_normal_ams(II)}{\thmref{thm:compact_normal_ams(II)}} give necessary and sufficient conditions for a normal $\mathsf{amst}$ to be compact. However, these results give no clue about how to construct a model satisfying a finitely satisfiable set $\Gamma$ using the models for its finite subsets. This, however, might be desirable sometimes. A different approach along the lines of \emph{\L o\'{s}’s theorem} (see \cite{ChangeKeisler1990,Los1955} for details) sheds some light on this question. A similar line of investigation is pursued in this section. The main tool in the following study is \emph{ultralimits}, introduced in \cite{Bernstein1970} for ultrafilters over the countable cardinal $\omega$, and later extended in \cite{Saks1978} for ultrafilters over arbitrary cardinals. Our treatment of these notions, however, closely follows \cite{Goldbring2022}, which defines ultralimits without reference to the cardinality of the index set. We briefly recall some definitions and theorems on ultrafilters and ultralimits before going into the technicalities.

Given a set $I$, a \emph{filter on $I$} is a set $\mathcal{F}\subseteq \mathcal{P}(I)$ satisfying the following properties. 
\begin{enumerate}[label=(\alph*)]
    \item $X\in \mathcal{F}$ and $Y\in \mathcal{F}$ implies that $X\cap Y\in \mathcal{F}$.
    \item $X\subseteq Y$ and $X\in \mathcal{F}$ implies that $Y\in \mathcal{F}$.
\end{enumerate}
A filter $\mathcal{F}$ on $I$ is said to be a \emph{proper filter} if $\emptyset\notin \mathcal{F}$. An \emph{ultrafilter on $I$} is a maximal proper filter on $I$. In other words, an ultrafilter on $I$ is a proper filter $\mathcal{U}$ on $I$ such that for every filter $\mathcal{F}\supsetneq \mathcal{U}$, $\emptyset\in \mathcal{F}$. 

A set $\mathcal{D}\subseteq \mathcal{P}(I)$ is said to satisfy the \emph{finite intersection property (FIP)} if, for all $n\in \mathbb{N}\setminus\{0\}$ and $D_1,\ldots,D_n\in\mathcal{D}$, $D_1\cap\cdots\cap D_n\neq\emptyset$. 

One can now easily check the following results.   

\hypertarget{thm:fip}{\begin{thm}{\label{thm:fip}}
Suppose $I$ is a set and $\mathcal{D}\subseteq\pow(I)$ such that $\mathcal{D}$ satisfies FIP. Then, the following collection of subsets of $I$ is a proper filter on $I$ containing $\mathcal{D}$.
\[
\{F\subseteq I\mid\,D_1\cap\cdots\cap D_n\subseteq F\hbox{ for some }D_1,\ldots,D_n\in\mathcal{D}\}
\]
This is, in fact, the smallest filter containing $\mathcal{D}$ and is called the \textit{filter generated by} $\mathcal{D}$.
\end{thm}}

\hypertarget{thm:ultrafilter_char}{\begin{thm}[\textsc{Properties of Ultrafilters}]{\label{thm:ultrafilter_char}}
    Suppose $I$ is a set and $\mathcal{U}$ a filter on $I$. Then, the following statements hold.
    \begin{enumerate}[label=(\arabic*)]
        \hypertarget{thm:ultrafilter_char(1)}{\item $\mathcal{U}$ is an ultrafilter on $I$ iff for all $A\subseteq I$, either $A\in \mathcal{U}$ or $I\setminus A\in \mathcal{U}$, but not both.}\label{thm:ultrafilter_char(1)}
        \hypertarget{thm:ultrafilter_char(2)}{\item If $\mathcal{U}$ is an ultrafilter on $I$ and $A_1,\ldots,A_n$ are
        subsets of $I$ such that $A_1 \cup\cdots\cup A_n \in \mathcal{U}$, then there exists $i \in \{1,\ldots,n\}$ such that $A_i \in \mathcal{U}$.\label{thm:ultrafilter_char(2)}} 
        \hypertarget{thm:ultrafilter_char(3)}{\item Given any proper filter $\mathcal{F}$ on $I$, there exists an ultrafilter $\mathcal{U}$ on $I$ such that $\mathcal{F}\subseteq\mathcal{U}$.\label{thm:ultrafilter_char(3)}}
    \end{enumerate}
\end{thm}}

\begin{defn}[\textsc{Ultralimit}]
    Suppose $(X,\tau)$ is a topological space and $(x_i)_{i\in I}$ a family in $X$. Let $\mathcal{U}$ be an ultrafilter on $I$. An element $x\in X$ is said to be a \emph{$\mathcal{U}$-ultralimit of $(x_i)_{i\in I}$ in $(X,\tau)$} if, for all $\tau$-open $U$ such that $x\in U$, $\{i\in I\mid\,x_i\in U\}\in \mathcal{U}$. In this case, we say that $(x_i)_{i\in I}$ \emph{$\mathcal{U}$-converges to $x$ in $(X,\tau)$}.
\end{defn}

We collect some basic results about ultralimits in the next theorem. 

\hypertarget{thm:ultralimit_basic}{\begin{thm}[\textsc{Properties of Ultralimits}]{\label{thm:ultralimit_basic}}
Suppose $(X,\tau)$ is a topological space. Let $\beta$ be a base and $\sigma$ a subbase for $\tau$.
    \begin{enumerate}[label=(\arabic*)]
        \hypertarget{thm:ultralimit_basic(1)}{\item $(X,\tau)$ is compact iff for any family $(x_i)_{i\in I}$ from $X$ and any ultrafilter $\mathcal{U}$ on $I$, $(x_i)_{i\in I}$ $\mathcal{U}$-converges in $(X,\tau)$. \label{thm:ultralimit_basic(1)}}
         \hypertarget{thm:ultralimit_basic(2)}{\item Let $(x_i)_{i\in I}$ be a family in $X$, $x\in X$ and $\mathcal{U}$ an ultrafilter on $I$. Then, the following statements are equivalent.\label{thm:ultralimit_basic(2)}}
        \begin{enumerate}[label=(2\alph*)]
            \item $x$ is a $\mathcal{U}$-ultralimit of $(x_i)_{i\in I}$ in $(X,\tau)$.
            \item For all $B\in \beta$, if $x\in B$, then $\{i\in I\mid\,x_i\in B\}\in \mathcal{U}$.
            \item For all $S\in \sigma$, if $x\in S$, then $\{i\in I\mid\,x_i\in S\}\in \mathcal{U}$.
        \end{enumerate}
    \end{enumerate}    
\end{thm}}
\begin{proof}
    For a proof of (1), see \cite[Theorem 3.1.9]{Goldbring2022}. We only prove (2).
    
    \underline{(2a) $\implies$ (2b)}: This is immediate, since $\beta\subseteq \tau$. 

    \underline{(2b) $\implies$ (2a)}: Let $(x_i)_{i\in I}$ be a family in $X$, $x\in X$ and $\mathcal{U}$ an ultrafilter on $I$ such that (2b) holds. Let $U\in \tau$ be such that $x\in U$. Now, since $\beta$ is a base for $\tau$, $U=\displaystyle\bigcup_{B\in\gamma}B$ for some $\gamma\subseteq\beta$. Now, by our hypothesis, $\{i\in I\mid\,x_i\in B\}\in \mathcal{U}$ for all $B\in\gamma$. Since $\{i\in I\mid\,x_i\in B\}\subseteq\displaystyle\bigcup_{B\in\gamma}\{i\in I\mid\,x_i\in B\}$ for any $B\in\gamma$ and $\mathcal{U}$ is a filter, $\displaystyle\bigcup_{B\in\gamma}\{i\in I\mid\,x_i\in B\}\in \mathcal{U}$. Now, $\{i\in I\mid\,x_i\in U\}=\displaystyle\bigcup_{j\in J}\{i\in I\mid\,x_i\in B_{j}\}$. Hence, $\{i\in I\mid\,x_i\in U\}\in \mathcal{U}$. 
    
    \underline{(2a) $\implies$ (2c)}: This is immediate, since $\sigma\subseteq \tau$. 
    
    \underline{(2c) $\implies$ (2a)}: Let $(x_i)_{i\in I}$ be a family in $X$, $x\in X$ and $\mathcal{U}$ an ultrafilter on $I$, such that (2c) holds. Moreover, let $\sigma_\beta=\left\{U\subseteq X\mid\,\displaystyle\bigcap_{S\in \xi}S= U~\text{for some finite}~\xi\subseteq\sigma\right\}$. Then, $\sigma_\beta$ is a base for $\tau$. We now show that the statement (2b) holds for the base $\sigma_\beta$. This would then imply (2a) by the equivalence of the statements (2a) and (2b) proved above.

    Let $B\in\sigma_\beta$. Then, $B=\displaystyle\bigcap_{S\in\xi}S$ for some finite $\xi\subseteq\sigma$. Now, by our hypothesis, $\{i\in I\mid\,x_i\in S\}\in \mathcal{U}$ for all $S\in \xi$. Since $\xi$ is finite and $\mathcal{U}$ is a filter, this implies that $\displaystyle\bigcap_{S\in\xi}\{i\in I\mid\,x_i\in S\}\in \mathcal{U}$. We note that $\displaystyle\bigcap_{S\in\xi}\{i\in I\mid\,x_i\in S\}=\left\{i\in I\mid\,x_i\in \displaystyle\bigcap_{S\in\xi} S\right\}=\{i\in I\mid\,x_i\in B\}$. Thus, $\{i\in I\mid\,x_i\in B\}\in \mathcal{U}$. So, the statement (2b) holds for the base $\sigma_\beta$ and hence, the statement (2a) holds.  
\end{proof}

\begin{defn}[\textsc{Ultramodel}]
Suppose $\mstr=(\mathbf{M},\models,\mathcal{P}(\lang))$ is a normal $\mathsf{amst}$ and $(\mathbf{M},\tau_N)$ the topological space as defined in \hyperlink{thm:compact_normal_ams(II)}{\thmref{thm:compact_normal_ams(II)}}. Let $(m_i)_{i\in I}$ be a family in $\mathbf{M}$ and $\mathcal{U}$ an ultrafilter on $I$. An element $\mathsf{u}\in \mathbf{M}$ is said to be a \emph{$\mathcal{U}$-ultramodel of $(m_i)_{i\in I}$} if $\mathsf{u}$ is a $\mathcal{U}$-ultralimit of $(m_i)_{i\in I}$ in $(\mathbf{M},\tau_N)$.
\end{defn}

\begin{rem}
    It is clear from the above definition that an ultramodel is a special ultralimit, when the family under consideration is from $\mathbf{M}$, where $(\mathbf{M},\models,\pow(\lang))$ is a normal $\amst$. As mentioned in the introductory section of this article, the elements of $\mathbf{M}$ can be thought of as models of subsets of $\lang$. So, we call ultralimits of families in $\mathbf{M}$, ultramodels.
\end{rem}

\begin{defn}[\textsc{Closed under Ultramodels}]
Suppose $\mstr=(\mathbf{M},\models,\mathcal{P}(\lang))$ is a normal $\mathsf{amst}$ and $(\mathbf{M},\tau_N)$ the topological space as before.
\begin{enumerate}[label=(\roman*)]
    \item A set $\mathbf{K}\subseteq \mathbf{M}$ is said to be \emph{closed under ultramodels with respect to $(m_i)_{i\in I}$ relative to an ultrafilter $\mathcal{U}$ on $I$}, if every $\mathcal{U}$-ultramodel of $(m_i)_{i\in I}$ in $(\mathbf{M},\tau_N)$ belongs to $\mathbf{K}$.
    \item A set $\mathbf{K}\subseteq \mathbf{M}$ is said to be \emph{closed under ultramodels with respect to $(m_i)_{i\in I}$} if it is closed under ultramodels with respect to $(m_i)_{i\in I}$ relative to every ultrafilter on $I$.
\end{enumerate}\end{defn}

Suppose $(m_i)_{i\in I}$ is a family in $\mathbf{M}$, the topological space defined in \hyperlink{thm:compact_normal_ams(II)}{\thmref{thm:compact_normal_ams(II)}}, and $\mathcal{U}$ an ultrafilter on $I$. Let $\mathsf{u}$ be a $\mathcal{U}$-ultramodel of $(m_i)_{i\in I}$ in $(\mathbf{M},\tau_N)$. 

Now, for any $\Sigma\subseteq\lang$, by \hyperlink{thm:Tarski_top}{\thmref{thm:Tarski_top}}, $\Mod(\Sigma)$ is $\tau_N$-closed, i.e., $\mathbf{M}\setminus\Mod(\Sigma)$ is $\tau_N$-open. So, $\mathsf{u}\in\mathbf{M}\setminus\Mod(\Sigma)$ implies that $\{i\in I\mid\,m_i\in\mathbf{M}\setminus\Mod(\Sigma)\}\in \mathcal{U}$. In other words, if $\{i\in I\mid\,m_i\notin \Mod(\Sigma)\}\notin \mathcal{U}$, then $\mathsf{u}\in\Mod(\Sigma)$. Now, as $\mathcal{U}$ is an ultrafilter on $I$, by \hyperlink{thm:ultrafilter_char(1)}{\thmref{thm:ultrafilter_char}(1)}, $\{i\in I\mid\,m_i\notin \Mod(\Sigma)\}\notin \mathcal{U}$ iff $I\setminus \{i\in I\mid\,m_i\notin \Mod(\Sigma)\}\in \mathcal{U}$, i.e., $\{i\in I\mid\,m_i\in \Mod(\Sigma)\}\in \mathcal{U}$. Thus, we can conclude that if $\mathsf{u}$ is a $\mathcal{U}$-ultramodel of $(m_i)_{i\in I}$ in $(\mathbf{M},\tau_N)$, then, for all $\Sigma\subseteq\lang$, $\{i\in I\mid\,m_i\in \Mod(\Sigma)\}\in \mathcal{U}$ implies that $\mathsf{u}\in \Mod(\Sigma)$, or equivalently, if $\{i\in I\mid\,m_i\models\Sigma\}\in \mathcal{U}$, then $\mathsf{u}\models\Sigma$.

Conversely, suppose, for all $\Sigma\subseteq\lang$, if $\{i\in I\mid\,m_i\models\Sigma\}\in\mathcal{U}$, then $\mathsf{u}\models\Sigma$. We claim that $\mathsf{u}$ is a $\mathcal{U}$-ultramodel of $(m_i)_{i\in I}$ in $(\mathbf{M},\tau_N)$.

Now, as described in \hyperlink{thm:compact_normal_ams(II)}{\thmref{thm:compact_normal_ams(II)}}, $\tau_N$ is generated by $\{\mathbf{M}\setminus\Mod(\{\alpha\})\mid\,\alpha\in\lang\}$ as subbase. Let $\alpha\in\lang$ and $\mathsf{u}\in\mathbf{M}\setminus\Mod(\{\alpha\})$. Then, $\mathsf{u}\notin\Mod(\{\alpha\})$, i.e., $u\not\models\{\alpha\}$. So, by our assumption, $\{i\in I\mid\,m_i\models\{\alpha\}\}\notin\mathcal{U}$. Since $\mathcal{U}$ is an ultrafilter on $I$, by \hyperlink{thm:ultrafilter_char(1)}{\thmref{thm:ultrafilter_char}(1)}, this implies that $I\setminus\{i\in I\mid\,m_i\models\{\alpha\}\}\in\mathcal{U}$, i.e., $\{i\in I\mid\,m_i\not\models\{\alpha\}\}\in\mathcal{U}$, or in other words, $\{i\in I\mid\,m_i\in\mathbf{M}\setminus\{\alpha\}\}\in\mathcal{U}$. Hence, by \hyperlink{thm:ultralimit_basic(2)}{\thmref{thm:ultralimit_basic}(2)}, we can conclude that $\mathsf{u}$ is a $\mathcal{U}$-ultramodel of $(m_i)_{i\in I}$ in $(\mathbf{M},\tau_N)$.

The above observation thus proves the following theorem.

\hypertarget{thm:loz<=}{\begin{thm}{\label{thm:loz<=}}
Suppose $\mstr=(\mathbf{M},\models,\mathcal{P}(\lang))$ is a normal $\mathsf{amst}$ and $(\mathbf{M},\tau_N)$ the topological space as described in \hyperlink{thm:compact_normal_ams(II)}{\thmref{thm:compact_normal_ams(II)}}. Let $(m_i)_{i\in I}$ be a family in $\mathbf{M}$, $\mathcal{U}$ an ultrafilter on $I$ and $\mathsf{u}\in \mathbf{M}$. Then, $\mathsf{u}$ is a $\mathcal{U}$-ultramodel of $(m_i)_{i\in I}$ in $(\mathbf{M},\tau_N)$ iff, for all $\Sigma\subseteq \lang$, $\{i\in I\mid\,m_i\models\Sigma\}\in \mathcal{U}$ implies that $\mathsf{u}\models\Sigma$.
\end{thm}}

A comparison of the above theorem with \L o\'{s}’s theorem \cite{ChangeKeisler1990,Los1955} here might help one appreciate the power of the notion of ultramodels. However, \L o\'{s}’s theorem relies heavily on the syntax of first-order logic, which we feel is unnecessary for this exposition. We instead deal with the simpler case of classical propositional logic (\textbf{CPL}) and describe the concept of \emph{ultravaluations}, introduced in \cite{KrzysztofBozena2023} to prove an analogue of \L o\'{s}'s theorem. As mentioned there, this is to some extent similar to \L o\'{s}'s \emph{ultraproduct}. It seems that the notion of ultravaluations, although not explicitly named or investigated in as much detail, had previous occurrences in literature (see, e.g., \cite{Paseau2010}).

We assume, as before, that $\lang_\mathbf{CPL}$ is the set of well-formed formulas (wffs) of \textbf{CPL} constructed in the usual way over a countable set of propositional variables $V$ using any adequate set of classical connectives, a truth assignment is a function $v:V\to\{0,1\}$, and that each truth assignment $v$ can be uniquely extended to a valuation $\overline{v}:\lang_\mathbf{CPL}\to\{0,1\}$. 

Let $(v_i)_{i\in I}$ be an indexed family of truth assignments, where $I\neq\emptyset$ and $\mathcal{U}$ an ultrafilter on $I$. Then, the \emph{ultravaluation} of
$(v_i)_{i\in I}$ under $\mathcal{U}$ is the unique extension of the truth assignment denoted by $\displaystyle\left(\mathop{\mu}_{i\in I}v_i\right)/\mathcal{U}$ and defined as follows. For each $p\in V$,
\[
\left[\left(\mathop{\mu}_{i\in I}v_i\right)/\mathcal{U}\right](p)=\begin{cases}
    1&\text{if}~\{i\in I\mid\,v_i(p)=1\}\in \mathcal{U}\\
    0&\text{else}
\end{cases}
\]
We next state, without proof, the following result, which is a version of \cite[Lemma 3.2]{KrzysztofBozena2023}.

\hypertarget{thm:ultravaluation}{\begin{thm}[\textsc{Ultravaluation Lemma}]{\label{thm:ultravaluation}}
    Suppose $(v_i)_{i\in I}$ is a family of truth assignments, where $I\neq\emptyset$, and $\mathcal{U}$ an ultrafilter on $I$. Let $\lang_{\mathbf{CPL}}$ be the set of well-formed formulas of \textbf{CPL}. Then, for any $\varphi\in\lang_{\mathbf{CPL}}$,
    \[
    \overline{\left[\left(\mathop{\mu}_{i\in I} v_i\right)/\mathcal{U}\right]}(\varphi)=1\quad\hbox{ iff }\quad\{i \in I\mid\,\overline{v}_i(\varphi)=1\} \in \mathcal{U}.
    \]
\end{thm}}

Let $\mathcal{V}$ be the set of all truth assignments and $\models\,\subseteq\mathcal{V}\times\pow(\lang_{\mathbf{CPL}})$ is defined as follows. For any $v\in\mathcal{V}$ and $\Gamma\subseteq\lang_{\mathbf{CPL}}$,
\[
v\models\Gamma\quad\hbox{ iff }\quad\overline{v}(\varphi)=1\hbox{ for all }\varphi\in\Gamma.
\]
Clearly, $(\mathcal{V},\models,\pow(\lang_{\mathbf{CPL}}))$ is a normal $\amst$. The conclusion of the above theorem can then be rephrased as follows.
\[
    \left[\left(\mathop{\mu}_{i\in I} v_i\right)/\mathcal{U}\right]\models\{\varphi\}\quad\hbox{ iff }\quad\{i \in I\mid\,v_i\models\{\varphi\}\} \in \mathcal{U}.
    \]

Our next criterion for the compactness of normal $\amst$s is as follows. 

\hypertarget{thm:finsat_ultramodel=>sat}{\begin{thm}[\textsc{Compactness for Normal $\mathsf{amst}$ (III)}]{\label{thm:finsat_ultramodel=>sat}}
Suppose $\mstr=(\mathbf{M},\models,\mathcal{P}(\lang))$ is a normal $\amst$ and $(\mathbf{M},\tau_N)$ the topological space as described in \hyperlink{thm:compact_normal_ams(II)}{\thmref{thm:compact_normal_ams(II)}}. Let $\Finset(\Sigma)$ denote the set of all finite subsets of $\Sigma\subseteq \lang$. Then, $\mstr$ is compact iff, for all finitely satisfiable $\Sigma\subseteq \lang$ and any ultrafilter $\mathcal{U}$ on $\Finset(\Sigma)$, every family $(m_{\Sigma_0})_{\Sigma_0\in \Finset(\Sigma)}$ $\mathcal{U}$-converges in $(\mathbf{M},\tau_N)$.
\end{thm}}

\begin{proof}
Suppose $\mstr$ is compact. Then, by \hyperlink{thm:ultralimit_basic(2)}{\thmref{thm:ultralimit_basic}(1)}, for every family $(x_i)_{i\in I}$ and any ultrafilter $\mathcal{U}$ on $I$, $(x_i)_{i\in I}$ $\mathcal{U}$-converges in $(\mathbf{M},\tau_N)$. So, in particular, for $I=\Finset(\Sigma)$, where $\Sigma\subseteq\lang$ is finitely satisfiable, and any ultrafilter $\mathcal{U}$ on $I$, every family $(m_{\Sigma_0})_{\Sigma_0\in I}$ $\mathcal{U}$-converges in $(\mathbf{M},\tau_N)$.

Conversely, suppose for any finitely satisfiable $\Sigma\subseteq\lang$ and any ultrafilter $\mathcal{U}$ on $\Finset(\Sigma)$, every family $(m_{\Sigma_0})_{\Sigma_0\in \Finset(\Sigma)}$ $\mathcal{U}$-converges in $(\mathbf{M},\tau_N)$. Let $\Sigma\subseteq\lang$ be finitely satisfiable. To establish that $\mstr$ is a compact $\amst$, we need to show that $\Sigma$ is satisfiable.

Since $\Sigma$ is finitely satisfiable, for each $\Sigma_0\in\Finset(\Sigma)$, there exists $m_{\Sigma_0}\in\mathbf{M}$ such that $m_{\Sigma_0}\models\Sigma_0$. We choose such an $m_{\Sigma_0}$, for each $\Sigma_0\in\Finset(\Sigma)$, and construct the family $(m_{\Sigma_0})_{\Sigma_0\in \Finset(\Sigma)}$. Now, for each $\Sigma_0\in\Finset(\Sigma)$, let $\overline{\Sigma_0}=\{\Gamma\in \Finset(\Sigma)\mid\,m_{\Gamma}\models \Sigma_0\}$. Since $\Sigma_0\in \overline{\Sigma_0}$ for each $\Sigma_0\in \Finset(\Sigma)$, $\overline{\Sigma_0}\ne\emptyset$ for all $\Sigma_0\in\Finset(\Sigma)$. We claim that $\{\overline{\Sigma_0}\mid\,\Sigma_0\in\Finset(\Sigma)\}\subseteq\pow(\Finset(\Sigma))$ satisfies FIP. 

Let $\overline{\Sigma^{(1)}},\ldots,\overline{\Sigma^{(n)}}\in\{\overline{\Sigma_0}:\Sigma_0\in\Finset(\Sigma)\}$. Then, $\Sigma^{(1)},\ldots,\Sigma^{(n)}\in\Finset(\Sigma)$. So, $\Sigma^{(1)}\cup\cdots\cup\Sigma^{(n)}\in\Finset(\Sigma)$. Hence, there exists $m_{\Sigma^{(1)}\cup\cdots\cup\Sigma^{(n)}}\in\mathbf{M}$ such that $m_{\Sigma^{(1)}\cup\cdots\cup\Sigma^{(n)}}\models\Sigma^{(1)}\cup\cdots\cup\Sigma^{(n)}$. Now, as $\mstr$ is normal, $m_{\Sigma^{(1)}\cup\cdots\cup\Sigma^{(n)}}\models\Sigma^{(i)}$ for each $1\le i\le n$. Thus, 
$\Sigma^{(1)}\cup\ldots\cup\Sigma^{(n)}\in \overline{\Sigma^{(1)}}\cap\cdots\cap \overline{\Sigma^{(n)}}$, which implies that $\overline{\Sigma^{(1)}}\cap\cdots\cap \overline{\Sigma^{(n)}}\neq\emptyset$. Thus, $\{\overline{\Sigma_0}\mid\,\Sigma_0\in\Finset(\Sigma)\}\subseteq\pow(\Finset(\Sigma))$ satisfies FIP and so, by \hyperlink{thm:fip}{\thmref{thm:fip}} and \hyperlink{thm:ultrafilter_char(3)}{\thmref{thm:ultrafilter_char}(3)}, there exists an ultrafilter, say $\mathcal{U}$, on $\Finset(\Sigma)$ containing $\{\overline{\Sigma_0}:\Sigma_0\in\Finset(\Sigma)\}$. Hence, by our hypothesis, the family $(m_{\Sigma_0})_{\Sigma_0\in \Finset(\Sigma)}$ $\mathcal{U}$-converges in $(\mathbf{M},\tau_N)$. Let $\mathsf{u}\in\mathbf{M}$ be a $\mathcal{U}$-ultramodel of $(m_{\Sigma_0})_{\Sigma_0\in \Finset(\Sigma)}$ in $(\mathbf{M},\tau_N)$. We claim that $\mathsf{u}\models \Sigma$. 

Since, for each $\Sigma_0\in \Finset(\Sigma)$, $\overline{\Sigma_0}=\{\Gamma\in\Finset(\Sigma)\mid\,m_\Gamma\models\Sigma_0\}\in \mathcal{U}$, by \hyperlink{thm:loz<=}{\thmref{thm:loz<=}}, $\mathsf{u}\models\Sigma_0$ for all $\Sigma_0\in \Finset(\Sigma)$. Then, as $\Sigma=\displaystyle\bigcup_{\Sigma_0\in\Finset(\Sigma)}\Sigma_0$ and $\mstr$ is normal, $\mathsf{u}\models\Sigma$. This implies that $\Sigma$ is satisfiable. Hence, $\mstr$ is compact.
\end{proof}

\begin{rem}{\label{rem:def_top}} Suppose $\mstr=(\mathbf{M},\models,\pow(\lang))$ is a normal $\amst$. Then, from the proof of \hyperlink{thm:cumulative(1)}{\thmref{thm:cumulative}(1)}, we know that $\Mod(\emptyset)=\mathbf{M}$. Thus, $\mathbf{M}\subseteq\displaystyle\bigcup_{\Gamma\subseteq\lang}\Mod(\Gamma)$. Moreover, by \hyperlink{thm:cumulative(3)}{\thmref{thm:cumulative}(3)}, for any $\Gamma_1,\Gamma_2\subseteq\lang$, $\Mod(\Gamma_1\cup\Gamma_2)\subseteq\Mod(\Gamma_1)\cap\Mod(\Gamma_2)$. Thus, the set $\{\Mod(\Gamma)\mid\,\Gamma\subseteq \lang\}$ forms a base on $\mathbf{M}$, which is clearly different from $\tau_N$, the topology described in \hyperlink{thm:compact_normal_ams(II)}{\thmref{thm:compact_normal_ams(II)}}. Let $\tau_C$ denote the topology generated by $\{\Mod(\Gamma)\mid\,\Gamma\subseteq\lang\}$.\end{rem}

We now obtain the following result for $(\mathbf{M},\tau_C)$ which is analogous to the one for $(\mathbf{M},\tau_N)$ proved in \hyperlink{thm:loz<=}{\thmref{thm:loz<=}}.

\hypertarget{thm:loz=>}{\begin{thm}{\label{thm:loz=>}}
Suppose $\mstr=(\mathbf{M},\models,\mathcal{P}(\lang))$ is a normal $\amst$ and $(\mathbf{M},\tau_C)$ the topological space defined above. Let $(m_i)_{i\in I}$ be a family in $\mathbf{M}$, $\mathcal{U}$ an ultrafilter on $I$ and $\mathsf{u}\in \mathbf{M}$. Then, $\mathsf{u}$ is an $\mathcal{U}$-ultralimit of $(m_i)_{i\in I}$ in $(\mathbf{M},\tau_C)$ iff, for all $\Sigma\subseteq \lang$, $\mathsf{u}\models\Sigma$ implies that $\{i\in I\mid\,m_i\models\Sigma\}\in \mathcal{U}$.
\end{thm}}

\begin{proof}
    Suppose $(m_i)_{i\in I}$ is a family in $\mathbf{M}$ and $\mathcal{U}$ an ultrafilter on $I$. Let $\mathsf{u}$ be a $\mathcal{U}$-ultralimit of $(m_i)_{i\in I}$ in $(\mathbf{M},\tau_C)$. Now, $\{\Mod(\Gamma)\mid\,\Gamma\subseteq\lang\}$ forms a base for $\tau_C$. Hence, by \hyperlink{thm:ultrafilter_char(2)}{\thmref{thm:ultrafilter_char}(2)}, for all $\Sigma\subseteq\lang$, if $\mathsf{u}\in\Mod(\Sigma)$, then $\{i\in I\mid\,m_i\in\Mod(\Sigma)\}\in\mathcal{U}$. In other words, for all $\Sigma\subseteq\lang$, if $\mathsf{u}\models\Sigma$, then $\{i\in I\mid\,m_i\models\Sigma\}\in\mathcal{U}$.

    Conversely, suppose for all $\Sigma\subseteq\lang$, $\mathsf{u}\models\Sigma$ implies that $\{i\in I\mid\,m_i\models\Sigma\}\in\mathcal{U}$, i.e., $\mathsf{u}\in\Mod(\Sigma)$ implies that $\{i\in I\mid\,m_i\in\Mod(\Sigma)\}\in\mathcal{U}$. Then, again as $\{\Mod(\Gamma)\mid\,\Gamma\subseteq\lang\}$ is a base for $\tau_C$, by \hyperlink{thm:ultrafilter_char(2)}{\thmref{thm:ultrafilter_char}(2)}, $\mathsf{u}$ is a $\mathcal{U}$-ultralimit of $(m_i)_{i\in I}$ in $(\mathbf{M},\tau_C)$.
\end{proof}

\begin{rem}
We note that \hyperlink{thm:loz<=}{\thmref{thm:loz<=}} and \hyperlink{thm:loz=>}{\thmref{thm:loz=>}} provide non-topological characterizations of ultralimits of families of elements of $\mathbf{M}$, where $(\mathbf{M},\models,\pow(\lang))$ is a normal $\amst$, in the topological spaces $(\mathbf{M},\tau_N)$ and $(\mathbf{M},\tau_C)$, respectively. This indicates that the material of this subsection could be developed from alternative definitions not involving the topological jargon. We have, however, chosen to go with the above presentation so as to tie up the results here with those in the previous subsection. This also makes the results of the two subsections readily comparable.
\end{rem}  

\hyperlink{thm:loz<=}{\thmref{thm:loz<=}} and \hyperlink{thm:loz=>}{\thmref{thm:loz=>}} can now be combined to provide a topological version of \L o\'{s}'s theorem.

\hypertarget{thm:topLoz}{\begin{thm}[\textsc{Topological Version of \L o\'{s}'s Theorem}]{\label{thm:topLoz}}
    Suppose $\mstr=(\mathbf{M},\models,\mathcal{P}(\lang))$ is a normal $\mathsf{amst}$ and $(\mathbf{M},\tau_N)$, $(\mathbf{M},\tau_C)$ the topological spaces as described in \hyperlink{thm:compact_normal_ams(II)}{\thmref{thm:compact_normal_ams(II)}} and \hyperlink{rem:def_top}{\remref{rem:def_top}}, respectively. Let $(m_i)_{i\in I}$ be a family in $\mathbf{M}$, $\mathcal{U}$ an ultrafilter on $I$ and $\mathsf{u}\in \mathbf{M}$. Then, $\mathsf{u}$ is a $\mathcal{U}$-ultralimit of  $(m_i)_{i\in I}$ in both $(\mathbf{M},\tau_N)$ and $(\mathbf{M},\tau_C)$ iff the following statement holds. For all $\Sigma\subseteq \lang$, $\mathsf{u}\models\Sigma$ iff $\{i\in I:m_i\models\Sigma\}\in \mathcal{U}$.
\end{thm}}

\begin{rem}
The similarity between the above theorem and \hyperlink{thm:ultravaluation}{\thmref{thm:ultravaluation}} is now clearly visible. \hyperlink{thm:ultravaluation}{\thmref{thm:ultravaluation}}, coupled with the above theorem, shows that ultravaluations are particular types of ultramodels. A similar argument using \L o\'{s}'s theorem and the above result will show that ultraproducts are special cases of ultramodels as well. It is, however, interesting to note that the notion of ultramodels, in spite of not yielding an `iff' condition as in \hyperlink{thm:ultravaluation}{\thmref{thm:ultravaluation}}, is still sufficient to prove the compactness of normal $\amst$s as in \hyperlink{thm:finsat_ultramodel=>sat}{\thmref{thm:finsat_ultramodel=>sat}}.
\end{rem}

The following corollary of \hyperlink{thm:topLoz}{\thmref{thm:topLoz}} provides us with a generalization of \hyperlink{thm:ultravaluation}{\thmref{thm:ultravaluation}} and of \L o\'{s}'s theorem.

\hypertarget{thm:genLoz}{\begin{thm}[\textsc{Generalized \L o\'{s}'s Theorem}]{\label{thm:genLoz}}
Suppose $\mstr=(\mathbf{M},\models,\mathcal{P}(\mathscr{L}))$ is a normal $\amst$ and $(\mathbf{M},\tau_C)$ is the topological space described in \hyperlink{rem:def_top}{\remref{rem:def_top}}. Let $(m_i)_{i\in I}$ be a family in $\mathbf{M}$ and $\mathcal{U}$ an ultrafilter on $I$. Suppose for all $\alpha$, there exists $\beta\in \mathscr{L}$ such that $\mathbf{M}\setminus\Mod(\{\alpha\})=\Mod(\{\beta\})$. Then $\mstr$ is compact iff $(\mathbf{M},\tau_C)$ is compact and the following condition is satisfied. For all $\Sigma\subseteq\mathscr{L}$ and for all $\mathcal{U}$-ultralimit $\mathsf{u}$ of $(m_i)_{i\in I}$ in $(\mathbf{M},\tau_C)$, $\mathsf{u}\models \Sigma$ iff $\{i\in I\mid\,m_i\models\Sigma\}\in \mathcal{U}$. 
\end{thm}}

\begin{proof} 
Suppose $\mathfrak{M}$ is compact. Then, by \hyperlink{thm:compact_normal_ams(II)}{\thmref{thm:compact_normal_ams(II)}}, $(\mathbf{M},\tau_N)$ is compact. We claim that $\tau_N=\tau_C$. 

Let $U\in \tau_N$. Then, $U=\displaystyle\bigcup_{k\in K}\left(\displaystyle\bigcap_{j=1}^{n_k}\mathbf{M}\setminus\Mod(\{\alpha^k_j\})\right)$, where $\alpha^k_j\in \mathscr{L}$, $K$ is some index set and for all $k\in K$, $n_k\in \mathbb{N}$. Now, by our hypothesis, for each $\alpha^k_j\in \mathscr{L}$, there exists $\beta^k_j\in \mathscr{L}$ such that $\mathbf{M}\setminus\Mod(\{\alpha^k_j\})=\Mod(\{\beta^k_j\})$. Thus, 
\[
U=\displaystyle\bigcup_{k\in K}\left(\displaystyle\bigcap_{j=1}^{n_k}\mathbf{M}\setminus\Mod(\{\alpha^k_j\})\right)=\displaystyle\bigcup_{k\in K}\left(\displaystyle\bigcap_{j=1}^{n_k}\Mod(\{\beta^k_j\})\right)=\displaystyle\bigcup_{k\in K}\Mod\left(\displaystyle\bigcup_{j=1}^{n_k}\{\beta^k_j\}\right),
\]
by \hyperlink{thm:cumulative(3)}{\thmref{thm:cumulative}(3)}, as $\mstr$ is normal. So, $U\in \tau_C$. Hence, $\tau_N\subseteq \tau_C$. 

On the other hand, $U\in \tau_C$ implies that $U=\displaystyle\bigcup_{k\in K}\Mod(\Gamma_k)$, where $K$ is some index set. Then, as $\mstr$ is compact,
\[
U=\displaystyle\bigcup_{k\in K}\Mod(\Gamma_k)=\displaystyle\bigcup_{k\in K}\left(\displaystyle\bigcup_{\Sigma\in\Finset(\Gamma_k)}\Mod(\Sigma)\right).
\]
Now, by \hyperlink{thm:cumulative(1)}{\thmref{thm:cumulative}(1)}, 
\[
\displaystyle\bigcup_{k\in K}\left(\displaystyle\bigcup_{\Sigma\in\Finset(\Gamma_k)}\Mod(\Sigma)\right)=\displaystyle\bigcup_{k\in K}\left(\displaystyle\bigcup_{\Sigma\in\Finset(\Gamma_k)}\left(\displaystyle\bigcap_{\alpha\in\Sigma}\Mod(\{\alpha\})\right)\right).
\]
Next, by our hypothesis, for each $\alpha\in \mathscr{L}$ there exists $\beta_\alpha\in\lang$ such that $\mathbf{M}\setminus\Mod(\{\alpha\})=\Mod(\{\beta_\alpha\})$. So,
\[
U=\displaystyle\bigcup_{k\in K}\left(\displaystyle\bigcup_{\Sigma\in\Finset(\Gamma_k)}\left(\displaystyle\bigcap_{\alpha\in\Sigma}\Mod(\{\alpha\})\right)\right)=\displaystyle\bigcup_{k\in K}\left(\displaystyle\bigcup_{\Sigma\in\Finset(\Gamma_k)}\left(\displaystyle\bigcap_{\alpha\in\Sigma}\mathbf{M}\setminus \Mod(\{\beta_\alpha\})\right)\right).
\]
Then, clearly, $U\in\tau_N$, since $\Sigma$ is finite. Thus, $\tau_C\subseteq \tau_N$. Hence, $\tau_N=\tau_C$. So, since $(\mathbf{M},\tau_N)$ is compact, $(\mathbf{M},\tau_C)$ is compact as well.

Now, let $\mathsf{u}$ be a $\mathcal{U}$-ultralimit of $(m_i)_{i\in I}$ in $(\mathbf{M},\tau_C)$. Then, by \hyperlink{thm:topLoz}{\thmref{thm:topLoz}}, $\mathsf{u}\models \Sigma$ iff $\{i\in I\mid\,m_i\models\Sigma\}\in \mathcal{U}$ for all $\Sigma\subseteq \mathscr{L}$. 

Conversely, suppose $(\mathbf{M},\tau_C)$ is compact and the following condition holds. For all $\Sigma\subseteq\lang$ and for all $\mathcal{U}$-ultralimit $\mathsf{u}$ of $(m_i)_{i\in I}$ in $(\mathbf{M},\tau_C)$, $\mathsf{u}\models\Sigma$ iff $\{i\in I\mid\,m_i\models\Sigma\}\in\mathcal{U}$. In particular, this means that for all $\Sigma\subseteq\mathscr{L}$, $\{i\in I\mid\,m_i\models\Sigma\}\in \mathcal{U}$ implies that $\mathsf{u}\models \Sigma$. So, by \hyperlink{thm:loz<=}{\thmref{thm:loz<=}}, $\mathsf{u}$ is a $\mathcal{U}$-ultralimit of $(m_i)_{i\in I}$ in $(\mathbf{M},\tau_N)$. Thus, for every family $(m_i)_{i\in I}$ in $\mathbf{M}$ and every ultrafilter $\mathcal{U}$ on $I$, $(m_i)_{i\in I}$ $\mathcal{U}$-converges in $(\mathbf{M},\tau_N)$. Then, by \hyperlink{thm:ultralimit_basic(1)}{\thmref{thm:ultralimit_basic}(1)}, $(\mathbf{M},\tau_N)$ is compact. Hence, $\mstr$ is compact, by \hyperlink{thm:compact_normal_ams(II)}{\thmref{thm:compact_normal_ams(II)}}. 
\end{proof}

We note that the converse part of the above proof holds for every normal $\amst$. Hence, we have the following sufficient condition for the compactness of normal $\amst$s.

\hypertarget{thm:sufcomp}{\begin{thm}[\textsc{Sufficient Condition for the Compactness of Normal $\amst$s}]{\label{thm:sufcomp}} 
Suppose $\mstr=(\mathbf{M},\models,\mathcal{P}(\mathscr{L}))$ is a normal $\amst$ and $(\mathbf{M},\tau_C)$ is the topological space described in \hyperlink{rem:def_top}{\remref{rem:def_top}}. Let $(m_i)_{i\in I}$ be a family in $\mathbf{M}$ and $\mathcal{U}$ an ultrafilter on $I$. If $(\mathbf{M},\tau_C)$ is compact and for all $\Sigma\subseteq\mathscr{L}$ and for all $\mathcal{U}$-ultralimit $\mathsf{u}$ of $(m_i)_{i\in I}$ in $(\mathbf{M},\tau_C)$, $\mathsf{u}\models \Sigma$ iff $\{i\in I\mid\,m_i\models\Sigma\}\in \mathcal{U}$, then $\mstr$ is compact.
\end{thm}} 
 
\hyperlink{thm:ultravaluation}{\thmref{thm:ultravaluation}} helps us obtain a proof of the compactness theorem for classical propositional logic (see \cite[Theorem 4.1]{KrzysztofBozena2023}). Similarly, \L o\'{s}'s theorem enables a proof of the compactness theorem for first-order predicate logic. Now, as \hyperlink{thm:genLoz}{\thmref{thm:genLoz}} is a generalization of \L o\'{s}’s theorem, it is natural to ask whether it can lead to the compactness of some class of $\amst$s \textit{without} the assumption of the previous theorem. We show below that ultramodels satisfying the `iff'-condition in \hyperlink{thm:genLoz}{\thmref{thm:genLoz}} can be used to characterize compact normal $\amst$s. We give a special name to such ultramodels.

\begin{defn}[\textsc{\L o\'{s}-limit/\L o\'{s}-model}]
    Suppose $\mstr=(\mathbf{M},\models,\pow(\lang))$ is a normal $\amst$. Let $(m_i)_{i\in I}$ be a family in $\mathbf{M}$ and $\mathcal{U}$ an ultrafilter on $I$. An element $\mathsf{l}\in \mathbf{M}$ is said to be a \emph{$\mathcal{U}$-\L o\'{s}-limit} or \emph{$\mathcal{U}$-\L o\'{s}-model} of $(m_i)_{i\in I}$ if, for all $\Sigma\subseteq\lang$, $\mathsf{l}\models\Sigma$ iff $\{i\in I:m_i\models\Sigma\}\in \mathcal{U}$.
\end{defn}

\begin{rem}
    We note that, by \hyperlink{thm:topLoz}{\thmref{thm:topLoz}}, a \L o\'{s}-model of a family in $\mathbf{M}$, where $(\mathbf{M},\models,\pow(\lang))$ is a normal $\amst$, is an ultralimit of the family simultaneously in two different topological spaces, viz., $(\mathbf{M},\tau_N)$ and $(\mathbf{M},\tau_C)$. 
\end{rem}

Suppose $\mstr=(\mathbf{M},\models,\pow(\lang))$ is an $\amst$. We define a relation $\preceq\,\subseteq\mathbf{M}\times\mathbf{M}$ as follows. For all $m,n\in \mathbf{M}$, $m\preceq n$ iff $\Th(\{m\})\subseteq \Th(\{n\})$. Clearly, $(\mathbf{M},\preceq)$ is a poset. We write $m\prec n$ if $m\preceq n$ and $m\ne n$. Finally, for any $m\in \mathbf{M}$, let $\mathbb{U}(m)=\{n\in \mathbf{M}\mid\,m\preceq n\}$. 

\hypertarget{thm:order<=>maxsat}{\begin{thm}{\label{thm:order<=>maxsat}}
    Suppose $\mstr=(\mathbf{M},\models,\mathcal{P}(\lang))$ is a normal $\mathsf{amst}$ and $m,n\in \mathbf{M}$. Then, $\Th(\{n\})$ is a maximal satisfiable set containing $\Th(\{m\})$ iff $n$ is a maximal element of $\mathbb{U}(m)$.
\end{thm}}
\begin{proof}
    Suppose $\Th(\{n\})$ is a maximal satisfiable set containing $\Th(\{m\})$, but $n$ is not a maximal element of $\mathbb{U}(m)$. Then, there exists $k\in \mathbb{U}(m)$ such that $n\prec k$. This implies that $\Th(\{m\})\subseteq\Th(\{n\})\subsetneq \Th(\{k\})$. However, since $k\models \Th(\{k\})$, i.e., $\Th(\{k\})$ is satisfiable, this contradicts the assumption that $\Th(\{n\})$ is a maximal satisfiable set containing $\Th(\{m\})$. Thus, $n$ is a maximal element of $\mathbb{U}(m)$.
    
    Conversely, suppose $n$ is a maximal element of $\mathbb{U}(m)$. So, $m\preceq n$, which implies that $\Th(\{m\})\subseteq\Th(\{n\})$. Now, as $n\models\Th(\{n\})$, $\Th(\{n\})$ is a satisfiable set containing $\Th(\{m\})$. If possible, suppose $\Th(\{n\})$ is not a maximal satisfiable set containing $\Th(\{m\})$. Then, there exists $\Gamma\subseteq\lang$ such that $\Gamma$ is satisfiable and $\Th(\{n\})\subsetneq\Gamma$. Since $\Gamma$ is satisfiable, there exists $k\in\mathbf{M}$ such that $k\models\Gamma$. Now, let $\alpha\in\Gamma$. Then, as $\mstr$ is normal, $k\models\{\alpha\}$. Thus, $\alpha\in\Th(\{k\})$. So, $\Gamma\subseteq\Th(\{k\})$. Hence, $\Th(\{m\})\subseteq\Th(\{n\})\subsetneq\Th(\{k\})$, which implies that $m\preceq n\prec k$. This, however, contradicts the assumption that $n$ is a maximal element of $\mathbb{U}(m)$. Thus, $\Th(\{n\})$ is a maximal satisfiable set containing $\Th(\{m\})$.
\end{proof}

\begin{defn}[\textsc{Pseudo-closure under \L o\'{s}-models}]
Suppose $\mstr=(\mathbf{M},\models,\mathcal{P}(\lang))$ is an $\amst$ and $I$ a set. Then, $\mathbf{K}\subseteq\mathbf{M}$ is said to be \emph{pseudo-closed under \L o\'{s}-models relative to $I$} if, for every family $(m_i)_{i\in I}$ of elements in $\mathbf{K}$ and every ultrafilter $\mathcal{U}$ on $I$, the following statement holds. If $\mathsf{l}$ is a $\mathcal{U}$-\L o\'{s}-model of $(m_i)_{i\in I}$, then $\mathbb{U}(\mathsf{l})\cap \mathbf{K}\ne\emptyset$.
\end{defn}

\hypertarget{thm:compact_normal_ams(IV)}{\begin{thm}[\textsc{Compactness for Normal $\amst$ (IV)}]{\label{thm:compact_normal_ams(IV)}}
    Suppose $\mstr=(\mathbf{M},\models,\mathcal{P}(\lang))$ is a normal $\mathsf{amst}$. For each $\Sigma\subseteq\lang$, let
    \[
    \modfin(\Sigma)=\{m\in \mathbf{M}\mid\,m\models \Sigma_0\hbox{ for some }\Sigma_0\in\Finset(\Sigma)\}.
    \]
    Moreover, suppose for any $\Sigma\subseteq\lang$, every family $(m_{\Sigma_0})_{\Sigma_0\in\Finset(\Sigma)}$ of elements in $\modfin(\Sigma)$ has a $\mathcal{U}$-\L o\'{s}-model, for any ultrafilter $\mathcal{U}$ on $\Finset(\Sigma)$. Then, $\mstr$ is compact iff for each $\Sigma\subseteq \lang$, $\modfin(\Sigma)$ is pseudo-closed under \L o\'{s}-models relative to $\Finset(\Sigma)$.
\end{thm}}
\begin{proof}
    Suppose $\mstr$ is compact. Let $\Sigma\subseteq \lang$ be such that $\modfin(\Sigma)$ is not pseudo-closed under \L o\'{s}-models relative to $\Finset(\Sigma)$. Then, there exists a family $(m_{\Sigma_0})_{\Sigma_0\in \Finset(\Sigma)}$ of elements in $\modfin(\Sigma)$ and an ultrafilter $\mathcal{U}$ on $\Finset(\Sigma)$ such that $(m_{\Sigma_0})_{\Sigma_0\in \Finset(\Sigma)}$ has a $\mathcal{U}$-\L o\'{s}-model $\mathsf{l}\in\mathbf{M}$ but $\mathbb{U}(\mathsf{l})\cap \modfin(\Sigma)=\emptyset$.
    
    Now, $\mathsf{l}\models\Th(\{\mathsf{l}\})$. Then, as $\mathsf{l}$ is a $\mathcal{U}$-\L o\'{s}-model of $(m_{\Sigma_0})_{\Sigma_0\in \Finset(\Sigma)}$, $\{\Sigma_0\in \Finset(\Sigma)\mid\,m_{\Sigma_0}\models\ \Th(\{\mathsf{l}\})\}\in \mathcal{U}$ as well. Since $\mathcal{U}$ is an ultrafilter, $\{\Sigma_0\in \Finset(\Sigma)\mid\,m_{\Sigma_0}\models\ \Th(\{\mathsf{l}\})\}\neq\emptyset$. Let $\Sigma_0\in \Finset(\Sigma)$ such that $m_{\Sigma_0}\models \Th(\{\mathsf{l}\})$. So, by normality of $\mstr$, $m_{\Sigma_0}\models\{\alpha\}$, i.e., $\alpha\in\Th(\{m_{\Sigma_0}\})$, for all $\alpha\in\Th(\{\mathsf{l}\})$. Thus, $\Th(\{\mathsf{l}\})\subseteq \Th(\{m_{\Sigma_0}\})$. Now, as $m_{\Sigma_0}\models \Th(\{m_{\Sigma_0}\})$, $\Th(\{m_{\Sigma_0}\})$ is satisfiable. So, by normality of $\mstr$, $\Th(\{m_{\Sigma_0}\})$ is finitely satisfiable. Since $\mstr$ is compact, this implies, by \hyperlink{thm:compact_normal_ams(I)}{\thmref{thm:compact_normal_ams(I)}}, that $\Th(\{m_{\Sigma_0}\})$ is contained in a maximal satisfiable set, say $\Delta$. Then, as $\Delta$ is satisfiable, there exists $v\in \mathbf{M}$ such that $v\models\Delta$. Again, by normality of $\mstr$, this implies that $v\models\{\delta\}$, i.e., $\delta\in\Th(\{v\})$, for all $\delta\in\Delta$. Thus, $\Delta\subseteq \Th(\{v\})$ and hence, $\Th(\{m_{\Sigma_0}\})\subseteq\Th(\{v\})$. Moreover, since $v\models\Th(\{v\})$, $\Th(\{v\})$ is satisfiable. So, $\Th(\{v\})$ is a satisfiable set containing $\Th(\{m_{\Sigma_0}\})$. Then, as $\Delta$ is a maximal satisfiable set containing $\Th(\{m_{\Sigma_0}\})$ and $\Delta\subseteq\Th(\{v\})$, $\Delta=\Th(\{v\})$. Thus, $\Th(\{\mathsf{l}\})\subseteq \Th(\{m_{\Sigma_0}\})\subseteq \Th(\{v\})$. Since $\Th(\{\mathsf{l}\})\subseteq\Th(\{v\})$, $\mathsf{l}\preceq v$ and hence, $v\in \mathbb{U}(\mathsf{l})$.
    
    Now, as $m_{\Sigma_0}\in\modfin(\Sigma)$, $m_{\Sigma_0}\models\Sigma^\prime_0$ for some $\Sigma^\prime_0\in \Finset(\Sigma)$. Since $\Th(\{m_{\Sigma_0}\})\subseteq \Th(\{v\})$, $v\models\Sigma^\prime_0$, as well, which implies that $v\in\modfin(\Sigma)$. Thus, $v\in \mathbb{U}(\{\mathsf{l}\})\cap\modfin(\Sigma)$. This is a contradiction. Hence, $\modfin(\Sigma)$ is pseudo-closed under \L o\'{s}-models relative to $\Finset(\Sigma)$.
    
    Conversely, suppose for any $\Sigma\subseteq \lang$, $\modfin(\Sigma)$ is pseudo-closed under \L o\'{s}-models relative to $\Finset(\Sigma)$.
    
    Let $\Sigma\subseteq\lang$ be finitely satisfiable. To establish that $\mstr$ is a compact $\amst$, we need to show that $\Sigma$ is satisfiable. Since $\Sigma$ is finitely satisfiable, for each $\Sigma_0\in\Finset(\Sigma)$, there exists $m_{\Sigma_0}\in\mathbf{M}$ such that $m_{\Sigma_0}\models\Sigma_0$. We choose such an $m_{\Sigma_0}$ for each $\Sigma_0\in \Finset(\Sigma)$. Then, for each $\Sigma_0\in \Finset(\Sigma)$, $m_{\Sigma_0}\in \modfin(\Sigma)$. Thus, $(m_{\Sigma_0})_{\Sigma_0\in\Finset(\Sigma)}$ is a family of elements in $\modfin(\Sigma)$. Then, as in the proof of \hyperlink{thm:finsat_ultramodel=>sat}{\thmref{thm:finsat_ultramodel=>sat}}, we define, for each $\Sigma_0\in\Finset(\Sigma)$, $\overline{\Sigma_0}=\{\Gamma\in \Finset(\Sigma)\mid\,m_{\Gamma}\models \Sigma_0\}$. Then, by the same arguments, there exists an ultrafilter, say $\mathcal{U}$, on $\Finset(\Sigma)$ containing $\{\overline{\Sigma_0}\mid\,\Sigma_0\in \Finset(\Sigma)\}$. 
    
    Now, by our hypothesis, the family $(m_{\Sigma_0})_{\Sigma_0\in\Finset(\Sigma)}$ has a $\mathcal{U}$-\L o\'{s}-model. Let $\mathsf{l}$ be a $\mathcal{U}$-\L o\'{s}-model of $(m_{\Sigma_0})_{\Sigma_0\in \Finset(\Sigma)}$. Then, as $\modfin(\Sigma)$ is pseudo-closed under \L o\'{s}-models relative to $\Finset(\Sigma)$, $\mathbb{U}(\mathsf{l})\cap \modfin(\Sigma)\ne\emptyset$. Let $k\in \mathbb{U}(\mathsf{l})\cap \modfin(\Sigma)$. Since $k\in \mathbb{U}(\mathsf{l})$, $\mathsf{l}\preceq k$, which implies that $\Th(\{\mathsf{l}\})\subseteq \Th(\{k\})$. We claim that $k\models\Sigma$. 
    
    Suppose the contrary, i.e., $k\not\models\Sigma$. Since $\mstr$ is a normal $\amst$, this implies that there exists $\varphi\in\Sigma$ such that $k\not\models\{\varphi\}$. Then, as $\Th(\{\mathsf{l}\})\subseteq \Th(\{k\})$, $\mathsf{l}\not\models\{\varphi\}$. Since $\mathsf{l}$ is a $\mathcal{U}$-\L o\'{s}-model of $(m_{\Sigma_0})_{\Sigma_0\in \Finset(\Sigma)}$, this implies that $\{\Sigma_0\in\Finset(\Sigma)\mid\,m_{\Sigma_0}\models\{\varphi\}\}\notin\mathcal{U}$. i.e., $\overline{\{\varphi\}}\notin \mathcal{U}$. This is a contradiction, as $\{\varphi\}\in\Finset(\Sigma)$ and $\mathcal{U}$ contains $\{\overline{\Sigma_0}\mid\,\Sigma_0\in\Finset(\Sigma)\}$. Thus, $k\models\Sigma$, i.e., $\Sigma$ is satisfiable. Hence, $\mstr$ is compact. 
\end{proof}

\section{Concluding Remarks}
In this article, we have proved several characterization theorems for the compactness of normal $\amst$s. We discuss below some directions in which further research on this topic may be carried out.

As we have noted earlier, \hyperlink{thm:compact_normal_ams(I)}{\thmref{thm:compact_normal_ams(I)}} is motivated by the Henkin-style proof of compactness, while \hyperlink{thm:compact_normal_ams(II)}{\thmref{thm:compact_normal_ams(II)}} is driven by topological considerations. Several other proofs of the compactness theorem can be found in the literature (see, e.g., \cite{Paseau2010}). It would be interesting to see if similar characterization theorems for compact normal $\mathsf{amst}$s can be obtained by generalizing these. Investigation along similar lines for arbitrary or other classes of $\mathsf{amst}$s also remains a goal for the future.

One might also consider generalizing the notion of compactness itself. Some such generalizations are available in the literature (see, e.g., \cite{MakowskyShelah1983,Paseau2010}). Each such concept would then, in turn, offer itself to further generalizations of compact $\mathsf{amst}$s. It seems that, with appropriate handling of technical subtleties, most of the results here can be lifted in a natural way. It would be interesting to work out the details. 

The notion of $\amst$s can be used to reformulate and generalize other fundamental theorems of model theory as well. For example, one can define connectives in a more general way than the one discussed in \cite{Garcia-Matos_Vaananen2005} and investigate corresponding versions of Lindström's theorem. 

\bibliographystyle{siam}
\bibliography{ref}

\end{document}